\documentclass[12pt,a4paper,reqno,final]{amsart} 
\usepackage{amsmath, amssymb, amsfonts, amsthm, verbatim, dsfont, graphicx, enumerate, pifont,tikz} 
\usepackage{hyperref}
\usepackage{tikz}
\usepackage{mathtools}
\usepackage{xcolor}
\usetikzlibrary{patterns}
\usepackage[notref,notcite]{showkeys} 
\usepackage[all,line,dvips]{xy}
\usepackage[obeyDraft]{todonotes} 
\usepackage[all]{xy}
\usepackage{amsaddr}
\theoremstyle{theorem}
\newtheorem{thm}{Theorem}

\newtheorem{prop}[thm]{Proposition}
\newtheorem{lemma}[thm]{Lemma}
\newtheorem{cor}[thm]{Corollary}
\newtheorem{question}[thm]{Question}

\theoremstyle{definition}
\newtheorem{defi}[thm]{Definition}
\newtheorem{rem}[thm]{Remark}

\newtheorem{ex}[thm]{Example}

\theoremstyle{plain}

\newcommand{\N}{\mathbb N} 
\newcommand{\Z}{\mathbb Z} 
\newcommand{\R}{\mathbb R} 
\newcommand{\C}{\mathbb C} 
\newcommand{\F}{\mathbb F}

\newcommand{\D}{\mathcal D}

\DeclareMathOperator{\ran}{rank}
\DeclareMathOperator{\Range}{Range}

\DeclareMathOperator{\Coker}{Coker}

\DeclareMathOperator{\supp}{supp}

\DeclareMathOperator{\id}{id}

\DeclareMathOperator{\ou}{Ou}
\DeclareMathOperator{\area}{Area}

\DeclareMathOperator{\image}{Image}

\DeclareMathOperator{\Symp}{Symp}

\DeclareMathOperator{\Graph}{graph}
\DeclareMathOperator{\Crit}{Crit}

\title{Topology of (small) Lagrangian cobordisms}
\author{Mads R. Bisgaard}
\thanks{In \cite{Bisgaard16} this paper was referred to under the name \emph{Classical topology of Lagrangian cobordisms}}
\address{Department of Mathematics, ETH Z{\"u}rich, R{\"a}mistrasse 101\\
8092 Z{\"u}rich,
Switzerland\\
\emph{\href{mailto:mads.bisgaard@math.ethz.ch}{mads.bisgaard@math.ethz.ch}}}
\begin{document}
\maketitle 
\begin{abstract}
We study the following quantitative phenomenon in symplectic topology: In many situations, if a Lagrangian cobordism is sufficiently \emph{small} (in a sense specified below) then its topology is to a large extend determined by its boundary. This principle allows us to derive several homological uniqueness results for small Lagrangian cobordisms. In particular, under the smallness assumption, we prove homological uniqueness of the class of Lagrangian cobordisms which, by Biran-Cornea's Lagrangian cobordism theory, induces operations on a version of the derived Fukaya category. We also establish a link between our results and Vassilyev's theory of Lagrange characteristic classes. Most currently known constructions of Lagrangian cobordisms yield \emph{small} Lagrangian cobordisms in many examples.  
\end{abstract}
\section{Introduction}
In this paper we study the topology of Lagrangian cobordisms connecting Lagrangian submanifolds of a symplectic manifold $(M^{2n},\omega)$. The idea of relating Lagrangian submanifolds by Lagrangian cobordisms was first conceived by Arnol'd \cite{Arnold80}. The idea has recently received a lot of attention, in part due to Biran-Cornea's work \cite{BiranCornea13}, \cite{BiranCornea14}. They showed (among many other things) that Lagrangian cobordisms provide a geometric realization of operations in a suitable version of the (derived) Fukaya category. They further showed that examples of such cobordisms arise as the trace of Lagrange surgery. It is therefore of interest to understand if all such cobordisms come from Lagrange surgery. More generally there are by now a few explicit constructions available for producing Lagrangian cobordisms. However, the topological and geometric nature imposed on a cobordism by requiring it admit a Lagrangian embedding into $\R^2 \times M$ remain rather mysterious. The present paper aims at exploring this nature. Some of the questions we attempt to answer are the following: How different can the topology of Lagrangian cobordant Lagrangians be? Does Lagrange surgery of two Lagrangians always give rise to a Lagrangian "trace of surgery"-cobordism? Is there a quantitative way to detect if a cobordism "originates" from Lagrange surgery?

\subsubsection*{Setting and notation}
$(M^{2n},\omega)$ will be assumed either closed or convex at infinity \cite{GromovEliashberg91}. Our Lagrangian cobordisms live in $\tilde{M}:=\R^2(x,y) \times M$ equipped with the symplectic structure $\tilde{\omega}:=\omega_{\R^2}\oplus \omega$, where $\omega_{\R^2}:=dx\wedge dy$. We denote by $\mathcal{L}=\mathcal{L}(M,\omega)$ the space of all closed, connected Lagrangian submanifolds of $(M,\omega)$. A Lagrangian cobordism $V\subset (\tilde{M},\tilde{\omega})$ relating two  ordered tuples $(L_i)_{i=1}^m, (L'_i)_{i=1}^{m'} \subset \mathcal{L}$ will always be assumed connected and is symbolically written
\[
V: (L'_i)_{i} \rightsquigarrow (L_j)_{j}.
\]
Viewing $V$ as an abstract cobordism its boundary $\partial V$ has a positive part and a negative part: $\partial_-V \approx \sqcup_{j=1}^mL_j$, $\partial_+V \approx \sqcup_{j=1}^{m'}L'_{j}$. When $V$ is oriented the $L_i$ and $L'_j$ inherit an orientation via the convention $\partial V=-\partial_- V \sqcup \partial_+ V$  (see Section \ref{seclagcob}).

\section{Main results}

Our first result is a cobordism version of the classical adjunction formula for Lagrangian submanifolds. Given oriented $L,L'\in \mathcal{L}$ we denote by $I(L,L')$ the intersection index of $(L,L')$ computed with respect to the orientation $\omega^n$ on $M$. 
\begin{thm}
\label{adcobthm}
Let $V: (L'_i)_{i}^{m'} \rightsquigarrow (L_j)_{j}^{m}$ be an oriented Lagrangian cobordism between two oriented ordered tuples $(L_i)_{i=1}^m, (L'_i)_{i=1}^{m'} \subset \mathcal{L}$. Then 
\begin{equation}
\label{eq0}
(-1)^{\tfrac{(n+1)n}{2}}\chi(V,\partial_- V)= \sum_{1\leq i<j\leq m} I(L_i,L_j) -\sum_{1\leq i<j\leq m'} I(L'_i,L'_j).
\end{equation}
In the non-oriented case this formula holds true modulo 2.
\end{thm}
The next result is in some sense the Floer-homological version of Theorem \ref{adcobthm}. It should be thought of as a relative version of the main result in Chekanov's beautiful paper \cite{Chekanov98}. To state it, we will say that a tuple $(L_i)_{i=1}^m\subset \mathcal{L}$ is \emph{transverse} if $L_i \pitchfork L_j$ for every $i\neq j$.

\begin{thm}
	\label{relChekanov}
	Let $(L'_i)_{i=1}^{m'},(L_j)_{j=1}^m \subset \mathcal{L}$ be two transverse tuples and let $V: (L'_i)_i \rightsquigarrow (L_j)_j$ be a small Lagrangian cobordism. If $V$ is spin then 
	\begin{equation}
	\label{eq3}
	\dim_{\F}\! H_*(V,\partial_{-}V; \F)\leq \sum_{1\leq i<j\leq m}\# (L_i\cap L_j)+\sum_{1\leq i<j\leq m'}\# ( L'_{i}\cap L'_{j}).
	\end{equation}
	for every field $\F$. If $V$ is not spin then (\ref{eq3}) still holds with $\F=\Z_2$.
\end{thm}

Let us explain the meaning of the word \emph{small} in the assumptions of this result. Denote by $A(\tilde{M},V)>0$ the \emph{bubbling threshold} of $V$. $A(\tilde{M},V)$ can intuitively be thought of as the area of the smallest non-constant holomorphic disk on $V$. For closed Lagrangian submanifolds this quantity was introduced by Chekanov \cite{Chekanov98}, but his definition easily generalizes to Lagrangian cobordisms (see Section \ref{seclagcob}). Recently Cornea and Shelukhin \cite{CorneaShelukhin15} introduced another non-negative quantity associated to $V$ - namely the so-called shadow of $V$, denoted by $\mathcal{S}(V)$. One can think of $\mathcal{S}(V)$ as measuring the "size" of the projection of $V$ to the $\R^2$-plane (see Section \ref{seclagcob}). In particular $\mathcal{S}(V)$ depends in a strong way on the embedding $V \hookrightarrow \tilde{M}$.  
\begin{defi}
	\label{def1}
	We say that a Lagrangian cobordism $V: (L'_i)_{i=1}^{m'} \rightsquigarrow (L_i)_{i=1}^m$ is \emph{small} if 
	\begin{equation}
	\label{eq35}
	\mathcal{S}(V)<A(\tilde{M},V).
	\end{equation}
\end{defi}

\begin{rem}
	A main novelty of Definition \ref{def1} is that it imposes no topological restrictions on $\partial V$. In fact every $L\in \mathcal{L}$ is the boundary component of a small Lagrangian cobordism (e.g. the trivial cobordism $\R \times L\subset (\tilde{M},\tilde{\omega})$). Moreover, most known constructions of Lagrangian cobordisms yield small cobordisms in many examples.
\end{rem}
\begin{rem}
	\label{rem1}
	Recall that if $(V^{n+1},\partial_+ V,\partial_- V)$ is a compact orientable cobordism and $\F$ denotes a field then Poincar\'e -Lefschetz duality gives $\F$-vector space isomorphisms 
	\[
	H_k(V,\partial_+V;\F)\cong H_{n+1-k}(V,\partial_-V;\F) \quad \forall \ k\in \Z.
	\]
	Of course, whether orientable or not, this always holds with $\F=\Z_2$. In particular we see that any compact cobordism $(V^{n+1},\partial_+ V,\partial_- V)$ satisfies $\chi(V,\partial_- V)=(-1)^{n+1}\chi(V,\partial_+ V)$.
\end{rem}

\subsection{Applications to elementary Lagrangian cobordisms}
\label{app1}
Theorems \ref{adcobthm} and \ref{relChekanov} are easiest to apply to \emph{elementary} Lagrangian cobordisms, i.e. Lagrangian cobordisms $V:L'\rightsquigarrow L$ which have just one negative and one positive end. For such $V$ the right-hand side of (\ref{eq3}) equals $0$. The following results all follow directly from this fact. For detailed proofs we refer to Section \ref{secproof}.

\begin{thm}
\label{cor2}
Let $L,L' \in \mathcal{L}$ and suppose at least one of them is spin. If $V: L'\rightsquigarrow L$ is a small Lagrangian cobordism then the inclusions $L,L'\hookrightarrow V$ induce isomorphisms on singular (co)homology. In particular, if there exists a small Lagrangian cobordism $V: L'\rightsquigarrow L$, then $H_*(L;\Z)\cong H_*(L';\Z)$ as graded groups and $H^*(L;\Z)\cong H^*(L';\Z)$ as graded rings.
If neither $L$ nor $L'$ is spin then the same result holds for homology with coefficients in $\Z_2$. 
\end{thm}
The following result is very much in the spirit of Chekanov's original result \cite{Chekanov98}. One can interpret it as saying that one cannot (geometrically) displace a Lagrangian by a small cobordism.
\begin{cor}
\label{cor1}
Let $L,L'\in \mathcal{L}_1$ and suppose at least one of them is spin. If there exists a small Lagrangian cobordism $V: L'\rightsquigarrow L$ then $L\cap L'\neq \emptyset$. Moreover, if $L \pitchfork L'$ then 
\begin{equation}
\label{eq33}
\dim_{\F}H_*(L;\F)\leq \#(L\cap L').
\end{equation}
for every field $\F$. Of course, if neither $L$ nor $L'$ is spin then the same conclusion holds for $\F=\Z_2$.
\end{cor}
\begin{cor}
\label{cor0}
No oriented $L\in \mathcal{L}$ with $\chi(L)\neq 0$ admits an oriented Lagrangian null-cobordism. Similarly, no $L\in \mathcal{L}$ admits a small Lagrangian null-cobordism.
\end{cor}
Of course the only case where the first conclusion in Corollary \ref{cor0} is a \emph{symplectic} phenomenon is when both $n$ and $\chi(L)$ are even.\footnote{Recall that every closed odd dimensional manifold $N$ satisfies $\chi(N)=0$. Similarly it is well-known that the boundary of a compact manifold has even Euler characteristic, so for $\chi(L)$ odd the conclusion of the corollary follows from classical topology.} The corollary in particular implies that the only oriented Lagrange surfaces in a symplectic 4-manifold which can be Lagrangian null-cobordant are Lagrangian tori. In contrast, recall that in the smooth category \emph{every} oriented surface is oriented null-cobordant!
A final application of Theorem \ref{cor2} to elementary Lagrangian cobordisms yields the following result which was pointed out to us by Fran\c{c}ois Charette.
\begin{cor}
	\label{whitehead}
	Let $L,L' \in \mathcal{L}$ be simply connected and suppose there exists a small, simply connected Lagrangian cobordism $V: L'\rightsquigarrow L$. If $\dim(V)\geq 6$ then $V$ is diffeomorphic to $[0,1]\times L$. In particular $L$ is diffeomorphic to $L'$.
\end{cor}
\begin{rem}
Corollary \ref{whitehead} concerns the case $\dim(M)\geq 10$. In the case $\dim(M)=4$ one can apply Theorem \ref{adcobthm} to get a conclusion of a similar spirit: Suppose $V:L\rightsquigarrow L'$ is an orientable Lagrangian cobordism between two orientable Lagrange surfaces $L,L'\in \mathcal{L}(M^4,\omega)$. Then, by additivity of $\chi$ and Theorem \ref{adcobthm}, we have $\chi(L)=\chi(V)=\chi(L')$. Hence $L$ and $L'$ are diffeomorphic!
\end{rem}
\begin{ex}
\label{ex1}
Consider $\mathbb{T}^2=\R^2/\Z^2$ equipped with the symplectic structure $\omega_{\mathbb{T}^2}$ inherited from $\R^2$. Denote by $L_h:=\{y=\tfrac{1}{2}\}$ and $L_v:=\{x=\tfrac{1}{2}\}$ the standard horizontal and vertical Lagrangians. Fix two curves $\gamma_1,\gamma_2 \subset \mathbb{T}^2$ as in Figure \ref{cobfig1} and denote by $\epsilon>0$ the sum of the areas of the little gray "triangles". Performing Lagrange surgery \cite{Polterovich91} along $\gamma_1$ we obtain the surgered Lagrangian $L_h \# L_v$. By Biran-Cornea's Lagrangian cobordism theory \cite{BiranCornea13} the trace of this surgery can be realized as a Lagrangian cobordism  $V_1:L_h \# L_v \rightsquigarrow (L_v,L_h)$. Similarly we can perform Lagrange surgery along $\gamma_2$ and obtain a Lagrangian cobordism $V_2:(L_v,L_h)\rightsquigarrow L_v \# L_h$. Concatenating $V_1$ and $V_2$ we obtain a Lagrangian cobordism $V:L_h \# L_v \rightsquigarrow  L_v \# L_h$. Denote by $B$ the bounded connected component of $\R^2 \backslash \pi (V)$, where $\pi: \tilde{M}\to \R^2$ denotes the projection. Consider now a split almost complex structure on $\tilde{\mathbb{T}}^2:=\R^2 \times \mathbb{T}^2$ of the type $\mathfrak{i}\oplus J$, where $\mathfrak{i}$ denotes the standard complex structure on $\R^2 \approx \C$. We then have an $\mathfrak{i}\oplus J$-holomorphic disk with boundary on $V$: 
\begin{align*}
u:(\overline{B},\partial \overline{B})&\to (\tilde{\mathbb{T}}^2,V) \\
z &\mapsto (z,(\tfrac{1}{2},\tfrac{1}{2})).
\end{align*}
Since the curve $u|_{\partial \overline{B}}\subset V$ projects to the non-trivial element of $H_1(\R^2 \backslash B;\Z_2)\cong \Z_2$ the inclusion $L_H \# L_V \hookrightarrow V$ does \emph{not} induce an isomorphism in $\Z_2$-homology. Hence, by Theorem \ref{cor2} $V$ is \emph{not} small. In fact it is easy to check that $[u]$ generates $\pi_2(\tilde{\mathbb{T}}^2,V)$ and therefore, if $\epsilon < \area(B)$, we conclude that the class $[u]\in \pi_2(\tilde{\mathbb{T}}^2,V)$ must contain a $\tilde{J}$-holomorphic disk for every $\tilde{\omega}$-compatible almost complex structure $\tilde{J}$ which is \emph{standard at} $\infty$ (see Section \ref{secbubcob}). This implies that $A(\tilde{\mathbb{T}}^2,V)=\int u^*\tilde{\omega}=\area(B)$ and thus
\begin{equation}
\label{eq34}
\mathcal{S}(V)=A(\tilde{\mathbb{T}}^2,V)+\epsilon.
\end{equation}
\begin{figure}
	\centering
	\begin{tikzpicture}
\draw [thick] (0,0) -- (0,5) -- (5,5) -- (5,0) -- (0,0);
\draw [thick] (2.5,0) -- (2.5,5);
\node [below] at (0.5,2.5) {\small $L_h$};
\draw [thick] (0,2.5) -- (5,2.5);
\node [right] at (2.5,0.5) {\small $L_v$};
\draw [thick, fill=gray] (1.5,2.5) to [out=0, in=270] (2.5,3.5) to (2.5,2.5) to (1.5,2.5);
\draw [thick, cyan] (1.5,2.5) to [out=0, in=270] (2.5,3.5);
\draw [thick, fill=gray] (2.5,3.5) to [out=270, in=180] (3.5,2.5) to (2.5,2.5) to (2.5,3.5);
\draw [thick, red] (2.5,3.5) to [out=270, in=180] (3.5,2.5);
\node at (2,3) {\small $\gamma_1$};
\node at (3,3) {\small $\gamma_2$};
\node [below] at (2.5,0) {\small $x$};
\node [left] at (0,2.5) {\small $y$};
\node [below] at (0,0) {\small $0$};
\node [below] at (5,0) {\small $1$};
\node [left] at (0,0) {\small $0$};
\node [left] at (0,5) {\small $1$};
\node [above right] at (0,0) {\small $\mathbb{T}^2$};
\draw [thick, dashed, <->] (6.5,1) to (6.5,0) to (7.5,0);
\draw [thick] (6.5,2.5) to (8,2.5) to [out=0, in=45] (8.5,2.8) to [out=45, in=135] (11.5,2.8) to [out=135, in=180] (12,2.5) to (13.5,2.5);
\draw [thick] (7.5,2.5) to (8,2.5) to [out=0, in=315] (8.5,2.2) to [out=315, in=225] (11.5,2.2) to [out=45, in=180] (12,2.5) to (12.5,2.5);
\draw [thick] (8.5,2.8) to [out=45, in=135] (11.5,2.8) to [out=315, in=45] (11.5,2.2) to [out=225, in=315] (8.5,2.2) to [out=135, in=225] (8.5,2.8);
\draw [thick, fill=black] (11.5,2.8) to [out=315, in=45] (11.5,2.2) to [out=45, in=180] (12,2.5) to [out=180, in=315] (11.5,2.8);
\draw [thick, fill=black] (8,2.5) to [out=0, in=315] (8.5,2.2) to [out=135, in=225] (8.5,2.8) to [out=225, in=0] (8,2.5);
\node at (10,2.5) {\small $B$};
\node [above right] at (6.5,0) {\small $\R^2$};
\node [above left] at (13.5,2.5) {\small $L_h \# L_v$};
\node [above right] at (6.5,2.5) {\small $L_v \# L_h$};
\node [below] at (10,1.6) {\small $\pi(V)$};
\end{tikzpicture}
	\caption{In the left figure $L_h,L_v\subset \mathbb{T}^2$ are indicated together with the curves $\gamma_1$ (blue) and $\gamma_2$ (red) along which Lagrange surgery is performed. The gray region has area $\epsilon$. In the right figure the projection $\pi(V)\subset \R^2$ of $V$ is indicated in black. As a consequence of the construction of Lagrangian cobordism via surgery \cite{BiranCornea13} we have $\area(\pi(V))=\epsilon$.}
	\label{cobfig1}
\end{figure}
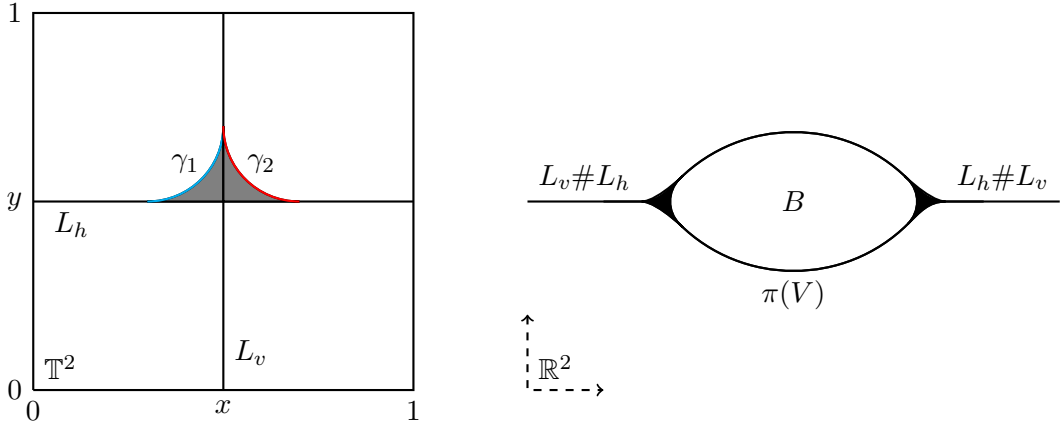
\end{ex}

The construction in Example \ref{ex1} can be carried out for every $\epsilon >0$, so (\ref{eq34}) implies that Theorem \ref{relChekanov} is optimal in the following sense: Its statement would cease to be true if one were to replace $A(\tilde{M},V)$ by a larger number (or $\mathcal{S}(V)$ by a smaller number) in (\ref{eq35}). We do not know, however, if Theorem \ref{relChekanov} continues to be true if one replaces "$<$" by "$\leq$" in (\ref{eq35}).

\subsection{Applications to Lagrangian cobordisms with multiple ends}
\label{app2}
Above we saw that, if we have a small Lagrangian cobordism $V: L' \rightsquigarrow L$, then the homology of $L$ determines that of both $V$ and $L'$. In this section we consider small Lagrangian cobordisms $V: L' \rightsquigarrow (L_i)_i$ from a "singleton" $L'\in \mathcal{L}$ to an ordered $m$-tuple $(L_i)_{i=1}^m\in \mathcal{L}$. The main interest in such cobordisms comes from the fact that, in certain situations, they are known to correspond to (possibly multiple) exact triangles in a suitable version of the derived Fukaya category \cite[Theorem A]{BiranCornea14}. The main questions we are interested in concern homological uniqueness: Does the data 
\begin{equation}
\label{eq105}
\bigoplus_{i=1}^mH_*(L_i;\Z)\quad \& \quad \mathcal{I}:=\sum_{1\leq i<j\leq m}\#(L_i\cap L_j)
\end{equation}
associated to $(L_i)_{i}$ determine $H_*(V;\Z)$ and $H_*(L';\Z)$? We first note an obstruction to finding small Lagrangian cobordisms with many ends:
\begin{cor}
\label{cormultend1}
Let $(L_i)_{i=1}^m\subset \mathcal{L}$ be a transverse $m$-tuple of Lagrangians in $(M,\omega)$. If $L'\in \mathcal{L}$ and there exists a small Lagrangian cobordism $V: L' \rightsquigarrow (L_i)_i$ then
\begin{equation}
\label{eq100}
\mathcal{I}\geq m-1.
\end{equation}
\end{cor}
Our first result in the direction outlined above is 
\begin{thm}
\label{cormultend4}
Let $(L_i)_{i=1}^m\subset \mathcal{L}$ be a transverse $m$-tuple such that every $L_i$ is spin and $\mathcal{I}=m-1$. Then every $L'\in \mathcal{L}$ for which there exists a small Lagrangian cobordism $V: L' \rightsquigarrow (L_i)_i$ satisfies 
\begin{equation*}
H_l(L';\Z)\cong \left\{
\begin{array}{ll}
\Z, & \text{if}\ l= 0,n \\
\oplus_{i=1}^m H_l(L_i;\Z), & \text{if}\ l\neq 0,n \\
\end{array}
\right.
\end{equation*}
and $V$ satisfies
\begin{equation*}
H_k(V;\Z)\cong 
\left\{
\begin{array}{ll}
\Z, & \text{if}\ k= 0 \\
\oplus_{i=1}^m H_k(L_i;\Z), & \text{if}\ k\neq 0,
\end{array}
\right.
\end{equation*}
where the isomorphism in the case $k\neq 0$ is induced by the inclusion $\sqcup_{i=1}^{m}L_i\hookrightarrow V$. Of course, if some $L_i$ is not spin then the same conclusion holds for homology with coefficients in $\Z_2$.
\end{thm}
Let's put this result into perspective.
\begin{defi}
We will say that an ordered $m$-tuple $(L_j)_{j=1}^m\subset \mathcal{L}$ is \emph{simple} if it is transverse and satisfies the following two conditions
\begin{enumerate}[a)]
\item
$L_i \cap L_j\cap L_k=\emptyset$ for all distinct $i,j,k$
\item
$(\cup_{j=1}^mL_j)\subset M$ is a connected subset.
\end{enumerate} 
Note that, if $(L_j)_{j=1}^m\subset \mathcal{L}$ is simple then all singular points of the Lagrange immersion $(\sqcup_j L_j)\hookrightarrow M$ are transverse and double.
\end{defi}

Biran-Cornea \cite{BiranCornea13} discovered that under certain conditions the trace of Lagrange surgery can be realized as Lagrangian cobordisms with multiple ends. In fact, as we will see below these "trace of surgery"-cobordisms are often small. Conversely, Theorem \ref{cormultend4} suggests that if $\mathcal{I}$ is not too large, then every small $V: L' \rightsquigarrow (L_i)_i$ is (homologically) the "trace of surgery"-cobordism of the $(L_i)_i$ and $L'$ is (homologically) a surgery of $(L_i)_i$. In order to explore this idea further we point out that in the present paper the term "Lagrange surgery" should be understood in the sense of \cite{Polterovich91}. Recall that, given a simple $m$-tuple $(L_j)_{j=1}^m\subset \mathcal{L}$, the operation developed in \cite{Polterovich91} allows one, after choosing an equipment at every singular point of the immersed Lagrangian $(\cup_i L_i)\subset (M,\omega)$, to paste in a Lagrange handle in order to obtain an embedded singleton $\widetilde{\#}_iL_i\in \mathcal{L}$.\footnote{We use $\widetilde{\#}$ in order to emphasize that, if the Lagrange immersion $(\cup_i L_i)\hookrightarrow (M,\omega)$ has multiple singular points, then the surgered Lagrangian $\widetilde{\#}_iL_i$ will not in general coincide with the connected sum $\#_iL_i$.} Although $\widetilde{\#}_iL_i$ in general depends on many choices, the diffeomorphism type of $\widetilde{\#}_iL_i$ only depends on the choice of an equipment of each singular point of $\cup_i L_i$ \cite{Polterovich91}. While there are no obstructions in the choice of equipment at each intersection point from the point of view of Lagrange surgery, the equipments must be chosen consistently in order to obtain an associated Lagrangian cobordism $\widetilde{\#}_iL_i \rightsquigarrow (L_i)_i$ (see Example \ref{ex4} below). The following result is perhaps the most important application of our results. 

\begin{cor}
\label{cormultend5}
Let $(L_i)_{i=1}^m\subset \mathcal{L}$ be a simple $m$-tuple whose intersection graph is a tree. Suppose in addition that every $L_i$ is spin. Then every small Lagrangian cobordism $V:L' \rightsquigarrow (L_i)_i$ from a singleton $L'\in \mathcal{L}$ to $(L_i)_i$ satisfies 
\[
H_*(V;\Z)\cong H_*(\widetilde{V};\Z),
\]
where $\widetilde{V}:\widetilde{\#}_iL_i \rightsquigarrow (L_i)_i$ is a Lagrangian "trace of surgery" cobordism, the surgery resulting in $\widetilde{\#}_iL_i$ being performed with respect to any equipment of $(L_i)_i$. Moreover, $L'$ satisfies 
\[
H_*(L';\Z)\cong H_*(\widetilde{\#}_iL_i;\Z).
\] 
If some $L_i$ isn't spin then these conclusions hold for homology with coefficients in $\Z_2$.
\end{cor}

\begin{ex}
\label{ex3}
There are many examples of symplectic manifolds $(M,\omega)$ which admit simple $m$-tuples $(L_j)_{j=1}^m\subset \mathcal{L}(M,\omega)$ whose intersection graphs are trees. One such example is the plumbing of $m$ unit codisk bundles of $m$ closed Riemannian manifolds. Other examples are $(A_m)$-configurations ($m>1$) of Lagrangian two-spheres in symplectic 4-manifolds, in the sense of \cite{Seidel99} (see also \cite[Section 8]{Seidel99} for explicit examples of such configurations inside $\{z_1^2+z_2^2=z_3^{m+1}+\tfrac{1}{2}\}\subset \C^3(z_1,z_2,z_3)$). Yet another example can be found in \cite{BiranMembrez16} where the authors (among many other things) study Lagrangian submanifolds of $\C P^2$ (symplectically) blown up at two points. Here they show the existence of two Lagrangian spheres having a single transverse intersection point.
\end{ex}
\begin{ex}
	\label{ex4}
	Fix $n>1$ and consider a simple pair $(L_1,L_2)\subset \mathcal{L}(M^{2n},\omega)$ of orientable Lagrangians. Suppose $\#(L_1\cap L_2)=k+1$ for some $k\in \N \cup \{0\}$. We equip the immersed Lagrangian $L_1\cup L_2\subset (M,\omega)$ consistently in the sense that in a Darboux-Weinstein neighborhood of every element of $L_1\cap L_2$, $L_1$ is identified with $\R^n$ and $L_2$ with $\mathfrak{i}\R^n$. Then fix some point $p\in L_1\cap L_2$ and prescribe that the equipment at $p$ is positive. This choice induces a sign to the equipment at every other element of $L_1\cap L_2$ (see \cite{Polterovich91}). We will say that $(L_1,L_2)$ is \emph{positive} if the sign of the equipment of every element of $L_1\cap L_2$ thus produced is positive, and we will say that it is \emph{negative} otherwise. This terminology does not depend on the choice of $p\in L_1\cap L_2$ and $(L_1,L_2)$ is positive if and only if $(L_2,L_1)$ is. Note in particular that if $k=0$, then the pair $(L_1,L_2)$ is automatically positive. If at least one of the Lagrangians in the simple pair $(L_1,L_2)\subset \mathcal{L}$ is non-orientable then we say that $(L_1,L_2)$ is negative.
	
	In case $(L_1,L_2)$ is a positive simple pair a construction due to Biran-Cornea \cite{BiranCornea13} yields a Lagrangian \emph{trace of surgery cobordism}
	\begin{equation}
	\label{eq106}
	V:\widetilde{\#}_iL_i \rightsquigarrow  (L_1, L_2) 
	\end{equation}
	where $\widetilde{\#}_iL_i \approx L_1\# L_2\# kP^n$ as a smooth manifold, with $P^n=S^{n-1}\times S^1$. If on the other hand $(L_1,L_2)$ is negative then the construction yields a Lagrangian trace of surgery cobordism as in (\ref{eq106}) but this time $\widetilde{\#}_i L_i \approx L_1\# L_2\# kQ^n$ as a smooth manifold, where $Q^n$ denotes the mapping torus of an orientation reversing involution of $S^{n-1}$. In either case the "trace of surgery"-cobordism $V$ has the homotopy type of the topological subspace $L_1\cup L_2\subset M$ and $\tilde{\omega}(\pi_2(\tilde{M},V))=\omega(\pi_2(M,L_1\cup L_2))$. It is easy to see from the construction of $V$ in \cite{BiranCornea13} that one can achieve $\mathcal{S}(V)<\delta$ for any $\delta>0$. In particular we see that $V$ can be made small if $\omega(\pi_2(M,L_1\cup L_2))=\delta \cdot \Z$ for some $\delta >0$.
\end{ex}

\subsubsection{Counting holomorphic disks}

We will use the terminology that a Lagrangian $L\in \mathcal{L}(M,\omega)$ is \emph{monotone} if 
\begin{equation*}
\omega|_{H_2(M,L;\Z)}\equiv \tau_L \cdot \mu_L|_{H_2(M,L;\Z)}
\end{equation*}
for some constant $\tau_L\geq 0$. Here $\omega |_{H_2(M,L;\Z)}$ denotes integration of $\omega$ and $\mu_L|_{H_2(M,L;\Z)}$ denotes the Maslov index. Hence, implicit in our definition of a monotone Lagrangian is that $\mu_L$ can be viewed as a homomorphism $H_2(M,L;\Z)\to \Z$. This is for example the case when $\pi_2(M,L)\cong H_2(M,L;\Z)$ or  $M=\R^{2n}$. Given a monotone Lagrangian $L\in \mathcal{L}(M,\omega)$ we denote by $\mathcal{D}_L\subset \pi_2(M,L)$ the set of elements which have Maslov index 2. Suppose now that $L$ is monotone and spin, and that a spin structure for $L$ has been fixed. View $\mathbb{D}:=\{z\in \C \ | \ |z|\leq 1\}$ as a Riemann surface with the complex structure induced from $\C$. For $\alpha \in \mathcal{D}_L$ and a $\omega$-compatible almost complex structure $J\in \mathcal{J}(M,\omega)$ we consider the moduli space
\[
\widetilde{\mathcal{M}}_L(\alpha,J):=\{ u:(\mathbb{D},\partial \mathbb{D})\to (M,L)\ | \ \overline{\partial}_J(u)=0, \ [u]=\alpha  \}.
\]
See e.g. \cite[Section (8f)]{Seidel08} for the definition of $\overline{\partial}_J$. For generic $J\in \mathcal{J}(M,\omega)$ the set $\widetilde{\mathcal{M}}_L(\alpha,J)$ admits the structure of a $(n+2)$-dimensional manifold, so the quotient $\mathcal{M}_L(\alpha,J)$ by the group of conformal transformation of the disk preserving $1\in \mathbb{D}$ is an $n$-dimensional oriented compact manifold (the orientation induced by the choice of a spin structure). For $\alpha \in \mathcal{D}_L$ we define $\eta_L(\alpha;\Z)\in \Z$ to be the degree of the evaluation map $ev:\mathcal{M}_L(\alpha,J)\to L$ given by $ev([u])=u(1)$. By the usual cobordism argument $\eta_L(\alpha;\Z)$ does not depend on the choice of $J$ and it depends on the choice of a spin structure on $L$ only up to a sign. If $L$ is orientable, but not spin then $\mathcal{M}_L(\alpha,J)$ need not be orientable, but it is still compact so the mod 2 degree $\eta_L(\alpha;\Z_2)\in \Z_2$ of $ev$ is well-defined. The following result first appeared in Chekanov's \cite{Chekanov97} with the assumption that the cobordism $V$ is monotone rather than small. However, the proof presented there seems to contain a gap. We do not know if the result as stated in \cite{Chekanov97} holds true but here we prove it under the stronger assumption that $V$ is small.
\begin{cor}
\label{cor6}
Let $L,L'\in \mathcal{L}(M,\omega)$ be monotone and spin and equipped with spin structures. Suppose moreover that $V:L' \rightsquigarrow L$ is a small Lagrangian cobordism. Then the isomorphism $H_*(L;\Z)\cong H_*(L';\Z)$ from Theorem \ref{cor2} induces a bijection $\mathcal{D}_L \leftrightarrow \mathcal{D}_{L'}$ such that the diagram
\[
\xymatrix{
	\mathcal{D}_L  \ar[rd]_{\eta_L(\cdot; \Z)} \ar@{<->}[rr] &  & \mathcal{D}_{L'} \ar[ld]^{\eta_{L'}(\cdot; \Z)} \\
	 & \Z & \\
}
\]
commutes up to sign.
\end{cor}
In \cite{Chekanov96} and \cite{Chekanov97} Chekanov found and studied his famous exotic Lagrangian tori in standard symplectic vector space. For each $k\in \{1,\ldots ,n\}$ Chekanov produced a monotone Lagrangian torus $T^n_k\in \mathcal{L}(\R^{2n},\omega_{\R^{2n}})$ with the property that $T^n_k$ and $T^n_{k'}$ are in different Hamiltonian isotopy classes whenever $k\neq k'$.\footnote{In this notation $T^n_n$ is the Lagrangian product torus consisting of the product of $n$ circles in $\R^{2n}=(\R^2)^n$. For the definition of $T^n_k$ we refer to \cite{Chekanov96} or \cite{Chekanov97}.} This result was proved in \cite{Chekanov97}, using ideas due to Eliashberg-Polterovich \cite{EliashbergPolterovich93}, by showing that there are exactly $k$ elements $\alpha \in \mathcal{D}_{T^n_k}$ for which $\eta_{T^n_k}(\alpha; \Z_2)=1$. A consequence of Corollary \ref{cor6} and \cite[Lemma 2.1]{Chekanov97} is that the same holds true for small cobordisms: $T^n_{k}$ and $T^n_{k'}$ are not cobordant by a small Lagrangian cobordism if $k\neq k'$. It was of course already known that there exist no orientable monotone Lagrangian cobordism $T^n_{k} \rightsquigarrow T^n_{k'}$ for $k$ odd and $k'$ even (see e.g. \cite[Remark 2.3.1.v.]{BiranCornea13}). 

We will now consider an explicit example using the $T^n_k$ in order to demonstrate that Lagrange surgery \emph{not} always gives rise to a (small) Lagrange cobordism. Given $a>0$ we denote by $T^n_k(a)$ Chekanov's torus $T^n_k$ embedded in $\R^{2n}$ in such a way that a Maslov 2 disk has area $a$. Suppose in the following that $n\geq 3$ is odd and consider $T:=T^n_k(a)$ for $k$ even and $T':=T^n_{k'}(a)$ for $k'$ odd. Now (Hamiltonian) perturb $T$ and $T'$ such that $T\pitchfork T'$ and $\#(T\cap T')=2$. Choose a compatible equipment of the tuple $(T,T')$. Clearly this equipment must be a negative. We denote by $T\widetilde{\#}_-T'$ the surgery performed with respect to this compatible equipment. Note that $T\widetilde{\#}_-T'$ is non-orientable and that we have have a "trace of surgery"-cobordism $T\widetilde{\#}_-T' \rightsquigarrow (T,T')$. Since we assume $n$ is odd we can switch the sign of the equipment at one of the intersection points by simply changing the choice of which torus is identified with $\R^n$ and which is identified with $\mathfrak{i}\R^n$ at that point. We denote by $T\widetilde{\#}_+T'$ the surgery performed with respect to this changed equipment. Note that $T\widetilde{\#}_+T'$ is orientable. 
\begin{cor}
\label{cor00}
Suppose in the above setting that we have a small Lagrangian cobordism $V:L\rightsquigarrow (T,T')$ for some $L\in \mathcal{L}(\R^{2n},\omega_{\R^{2n}})$. Then both $L$ and $V$ are non-orientable and 
\[
H_*(L;\Z_2)\cong H_*(T\widetilde{\#}_-T';\Z_2).
\]
Moreover, 
\[
H_*(V;\Z_2)\cong H_*(\widetilde{V};\Z_2)
\]
where $\widetilde{V}:T\widetilde{\#}_-T'\rightsquigarrow (T,T')$ denotes the trace of surgery Lagrangian cobordism. In particular there does not exist a small Lagrangian cobordism $T\widetilde{\#}_+T' \rightsquigarrow (T,T')$.
\end{cor}

\begin{rem}[Kazaryan-Vassilyev characteristic classes]
\label{remkazvas}
In general, given $(L_i)_{i}\subset \mathcal{L}$ with $\mathcal{I}\geq m$, one cannot expect that the data (\ref{eq105}) determines the homology of every $L'\in \mathcal{L}$ for which there exists a small Lagrangian cobordism $V: L' \rightsquigarrow (L_i)_i$. However, in certain situations one can deduce information about $H_*(L')$ if additional information about (Lagrange) characteristic classes of $L'$ is known.
Consider a closed manifold $W^n$ with cotangent bundle $T^*W\stackrel{\pi_W}{\longrightarrow} W$, and denote by $\omega$ the canonical (exact) symplectic form on $T^*W$. Recall that a \emph{caustic} of a Lagrangian $L\subset (T^*W,\omega)$ is a singularity of the map $\pi_W|_L:L\to W$. Arnol'd observed that there are topological obstructions to the coexistence of different types of caustics for a single Lagrangian $L\subset (T^*W,\omega)$ \cite{Vassilyev88}. In his beautiful book \cite{Vassilyev88} Vassilyev introduced Lagrange characteristic classes in the cohomology ring of a Lagrangian in $(T^*W,\omega)$ which "measure" these obstructions (see also \cite{Vassilyev81} and \cite[Chapter 6, Section 3.3-3.4]{ArnoldGivental01}). Later Kazaryan \cite{Kazaryan951}, \cite{Kazaryan952} found additional Lagrange characteristic classes corresponding to what he called "hidden singularities" and developed the theory of Lagrange characteristic classes in greater generality. Theorem \ref{relChekanov} shows that the existence of a small Lagrangian cobordism imposes even further restrictions on the Lagrange characteristic classes of its ends than the ones coming from purely homological reasons. More precisely we have the following estimate (here "Lagrange characteristic classes" should be understood in the sense of either \cite{Vassilyev88} or \cite{Kazaryan952}):

\begin{prop} 
\label{propkazvas}
Let $(L_i)_{i=1}^m\subset \mathcal{L}(T^*W, \omega)$ be an $m$-tuple of Lagrangians and suppose there are exactly $k\in \N \cup \{0\}$ Lagrange characteristic classes, each of which is non-zero in some $H^*(L_i;\Z_2 )$. Suppose $L'\in \mathcal{L}(T^*W,\omega)$ is a Lagrangian for which there exists a small Lagrangian cobordism $V:L'\rightsquigarrow (L_i)_i$. Then $L'$ has at most $\mathcal{I}+k+1-m$ distinct non-zero Lagrange characteristic classes in $H^{\geq 1}(L';\Z_2)$.
\end{prop}
Note that in case $\mathcal{I}=m-1$ the proposition says that $L'$ has \emph{at most} $k$ non-vanishing Lagrange characteristic classes in $H^{\geq 1}(L';\Z_2)$. In fact in this case we know (by Theorem \ref{cormultend4}) that it has exactly $k$ non-vanishing Lagrange characteristic classes in $H^{\geq 1}(L';\Z_2)$.

Given a tuple $(L_i)_{i=1}^m\subset \mathcal{L}(T^*W, \omega)$, one can apply the estimate in Proposition \ref{propkazvas} and Kazaryan-Vassilyev's theory to obtain information about the caustics of a Lagrangian $L'\in \mathcal{L}(T^*W,\omega)$ for which there exists a small Lagrangian cobordism $V:L'\rightsquigarrow (L_i)_i$. The proposition is particularly nice if every $L_i$ is an exact perturbation of the zero-section $W\subset T^*W$, because in this case one has $k=0$. As an example, consider $S^n=\{x=(x_1,\ldots ,x_{n+1})\in \R^{n+1}\ |\ |x|=1 \}$ for $n\leq 6$ and let $f:S^n\to \R$ be the height function $f(x)=x_{n+1}$. Denote by $L_1$ the zero section in $(T^*S^n,\omega)$ and define $L_2:=\phi_F^1(L_1)$, where $F:=\pi_{S^n}^*f\in C^{\infty}(T^*S^n)$. Then $\mathcal{I}=m=2$ and $k=0$ for the pair $(L_1,L_2)\subset \mathcal{L}(T^*S^n,\omega)$. Hence, if $L'\in \mathcal{L}(T^*S^n,\omega)$ is a Lagrangian for which there exists a small Lagrangian cobordism\footnote{By Example \ref{ex4} such a Lagrangian is obtained by performing Lagrange surgery on the tuple $(L_1,L_2)$.} $V:L'\rightsquigarrow (L_1,L_2)$, then $L'$ can have at most one non-trivial Lagrange $\Z_2$-characteristic class of degree $\geq 1$. So if in addition it is known that $L'$ is nonorientable, then Vassilyev's theory \cite{Vassilyev88} implies that all singularities of the map $\pi_{S^n}|_{L'}:L'\to S^n$ of codimension $>1$ are $\Z_2$-homologically trivial, in the sense that their associated characteristic classes vanish.
\end{rem}
\begin{rem}
Consider for $\delta>0$ the space $\mathcal{L}_{\delta}:=\{ L\in \mathcal{L}\ |\ \omega(\pi_2(M,L))=\delta \cdot \Z \}$. In \cite{CorneaShelukhin15} Cornea and Shelukhin defined $d_c:\mathcal{L}_{\delta}\times \mathcal{L}_{\delta} \to [0,\infty]$ by 
\[
d_c(L,L'):=\inf_V(\mathcal{S}(V)),
\]
where the infimum runs over all Lagrangian cobordisms $V:L'\rightsquigarrow L$ satisfying $\tilde{\omega}(\pi_2(\tilde{M},V))=\delta \cdot \Z$. They showed that $d_c$ defines a (non-degenerate) metric on $\mathcal{L}_{\delta}$. From this point of view Theorem \ref{relChekanov} says that $H_*(L;\Z_2)\cong H_*(L';\Z_2)$ if $L,L'\in \mathcal{L}_{\delta}$ satisfy $d_c(L,L')<\delta$.
\end{rem}

\begin{rem}
	Consider the subgroup $\Symp_{c|M}(\tilde{M},\tilde{\omega})\leq \Symp(\tilde{M},\tilde{\omega})$ of symplectomorphisms which are \emph{compactly supported relative to $M$}. $\Symp_{c|M}(\tilde{M},\tilde{\omega})$ consists, by definition, of the $\psi \in \Symp(\tilde{M},\tilde{\omega})$ for which there exist a compact subset $C\subset \R^2$ and a $\psi'\in \Symp(M,\omega)$ such that $\psi=\id \times \psi'$ on $(\R^2 \backslash C)\times M$. As will be clear from Definition \ref{def2} below, $A(\tilde{M},V)$ it is invariant under elements of $\Symp_{c|M}(\tilde{M},\tilde{\omega})$. Hence, in the situation of Example \ref{ex1} we conclude that 
	\[
	\mathcal{S}(\psi(V))\geq \area(B) \quad \forall \ \psi \in \Symp_{c|M}(\tilde{\mathbb{T}}^2,\tilde{\omega}).
	\]
	This can be viewed as a kind of non-squeezing statement. More generally we obtain the following "Lagrangian non-squeezing statement": If $V: L'\rightsquigarrow L$ is a Lagrangian cobordism satisfying $H_*(V,L;\Z_2)\neq 0$, then Corollary \ref{cor2} implies that
	\[
	\mathcal{S}(\psi(V))\geq A(\tilde{M},V) \quad \forall \ \psi \in \Symp_{c|M}(\tilde{M},\tilde{\omega}) .
	\]
	See also \cite{Bisgaard16} for a different Lagrangian cobordism non-squeezing result. 
\end{rem}

\begin{rem}
	The main observation in the proof of Corollary \ref{cormultend1} can also be used to obtain the following estimate: If $V: (L'_i)_{i=1}^{m'} \rightsquigarrow (L_i)_{i=1}^{m}$ is a small Lagrangian cobordism and both $(L'_i)_{i=1}^{m'}$ and $(L_i)_{i=1}^{m}$ are transverse then 
	\[
	\max(m,m')-1\leq \sum_{1\leq i<j\leq m}\# (L_i\cap L_j)+\sum_{1\leq i<j\leq m'}\# ( L'_{i}\cap L'_{j}).
	\] 
\end{rem}

\subsection{Outlook and questions}
To the author's knowledge there a currently three known explicit constructions of elementary Lagrangian cobordisms (up to concatenating Lagrangian cobordisms coming from these constructions and applying symplectomorphisms of course): 
\begin{enumerate}[a)]
	\item
	The Lagrangian suspension construction (see Section 3.1.E in \cite{Polterovich01}).
	\item
	Lagrangian antisurgery. This construction was recently introduced by Haug \cite{Haug15}.
	\item
	Concatenating multi-ended Lagrangian cobordisms which are constructed as the trace of Lagrange surgery (this construction is due to Biran-Cornea \cite{BiranCornea13}).
\end{enumerate}
Elementary Lagrangian cobordisms of type c) are never small (see Example \ref{ex1}) and we do not know of any examples where a Lagrangian cobordism of type b) is small. On the other hand there are many examples of small Lagrangian suspensions.
\begin{question}
Is every small elementary Lagrangian cobordism $V\subset (\tilde{M},\tilde{\omega})$ the image of a Lagrangian suspension under an element of $\Symp_{c|M}(\tilde{M},\tilde{\omega})$?
\end{question} 
This question is very closely related to a conjecture by Biran-Cornea which states that every exact Lagrangian cobordism is Hamiltonian isotopic to a Lagrangian suspension \cite{Suarez14}.
Although this conjecture remains unsolved both Su\'arez \cite{Suarez14} and Tanaka \cite{Tanaka14} have made good progress towards confirming it. 

To our knowledge there is only one known explicit construction which produces Lagrangian cobordisms of the type $L'\rightsquigarrow (L_i)_{i}$ for an $m$-tuple $(L_i)_{i=1}^m\subset \mathcal{L}$ and $L'\in \mathcal{L}$ and that is the Biran-Cornea trace of surgery cobordism \cite{BiranCornea13}. As was noted in Example \ref{ex4} such cobordisms can often be made small. Motivated by Corollary \ref{cormultend5} and \ref{cor00} we ask:
\begin{question}
Given a simple $m$-tuple $(L_i)_{i=1}^m\subset \mathcal{L}$ whose intersection graph is a tree as well as $L'\in \mathcal{L}$, is every small Lagrangian cobordism $L'\rightsquigarrow (L_i)_{i}$ the image under an element of $\Symp_{c|M}(\tilde{M},\tilde{\omega})$ of a trace of surgery cobordism coming from Lagrange surgery of $(L_i)_{i}$?
\end{question}
\subsection*{Acknowledgement}
I am extremely grateful to my advisor Paul Biran for all the advice and encouragement during my work on this project. I am indebted to Luis Haug for his interest in this work and for pointing out a few mistakes in an early draft of this paper. I am indebted to Fran\c{c}ois Charette for pointing out that Corollary \ref{whitehead} follows directly from Theorem \ref{cor2}. I thank Octav Cornea for encouraging me to study \cite{Suarez14}.
\section{Preliminaries on Lagrangian cobordisms}
\label{seclagcob}
Here we collect a few facts and definitions about Biran-Cornea's Lagrangian cobordism theory \cite{BiranCornea13}, \cite{BiranCornea14}. We also give precise definitions of the objects used above. Given subsets $V\subset \tilde{M}=\R^2 \times M$ and $U\subset \R^2$ we write $V|_{U}=V\cap \pi^{-1}(U)$, where $\pi:\tilde{M}\to \R^2$ denotes the natural projection. Given an oriented manifold $K^k$ with boundary $\partial K$ we use the convention that the induced boundary orientation of $\partial K$ is given by the "outward normal first" convention. I.e. if $q\in \partial K$ then $(v_1,\ldots ,v_{k-1})$ is an oriented basis for $T_q(\partial K)$ if $(n_q,v_1,\ldots ,v_{k-1})$ is an oriented basis for $T_qK$, where $n_q\in T_qK$ points outward from $K$.

\begin{defi}[\cite{BiranCornea13}]
\label{defcob1}
We say that two ordered tuples $(L_i)_{i=1}^{m},(L'_i)_{i=1}^{m'}\in \mathcal{L}$ are \emph{Lagrangian cobordant} if for some $R>0$ there exists a smooth compact Lagrangian submanifold $V\subset ([-R,R]\times \R \times M, \omega_{\R^2}\oplus \omega )$ with boundary $\partial V =V\cap (\{\pm R\}\times \R \times M)$ satisfying the condition that for some $\epsilon >0$ we have
\begin{align}
\label{cob21}
V|_{[-R,-R+\epsilon)\times \R}&=\bigsqcup_{i=1}^{m}([-R,-R+\epsilon)\times \{i\})\times L_i \\
\label{cob22}
V|_{(R-\epsilon,R]\times \R}&=\bigsqcup_{j=1}^{m'}((R-\epsilon,R]\times \{j\})\times L'_j.
\end{align}
In particular $V$ defines a smooth compact cobordism $(V,\bigsqcup_{i=1}^{m}L_i,\bigsqcup_{j=1}^{m'}L'_j)$. We write $V:(L'_j)_j \rightsquigarrow (L_i)_i$. In case each $L_i$ and each $L'_j$ is oriented we say that $V$ is an \emph{oriented Lagrangian cobordism} if $V$ carries an orientation such that the associated boundary orientation of $\partial V$ coincides with the orientation given by $\partial V=( -\bigsqcup_{i=1}^{m}L_i) \bigsqcup ( \bigsqcup_{j=1}^{m'}L'_j)$. 
\end{defi}

As is customary in the field our notation does not distinguish between a Lagrangian cobordism and its horizontal $\R$-extension. This extension is a \emph{Lagrangian with cylindrical ends}. More generally we have
\begin{defi}[\cite{BiranCornea13}]
A \emph{Lagrangian with cylindrical ends} is a boundaryless Lagrangian submanifold $V\subset (\tilde{M},\tilde{\omega})$ satisfying the conditions that 1) $V|_{[a,b]\times \R}$ is compact for all $a<b$ and 2) there exists $R>0$ such that 
\begin{align*}
V|_{(-\infty,-R]\times \R}&=\bigsqcup_{i=0}^{m}((-\infty,-R]\times \{a^-_i\})\times L_i \\
V|_{[R,\infty)\times \R}&=\bigsqcup_{j=0}^{m'}([R,\infty)\times \{a^+_j\})\times L'_j
\end{align*}
for Lagrangians $L_i,L'_j \subset (M,\omega)$ and constants $a^-_i,a^+_j\in \R$ verifying $a^-_i\neq a^-_{i'}$ for $i\neq i'$ and $a^+_j\neq a^+_{j'}$ for $j\neq j'$.
\end{defi} 
\subsubsection{The shadow and bubbling threshold of a Lagrangian cobordism}
\label{secbubcob}
Given a Lagrangian with cylindrical ends $V\subset (\tilde{M},\tilde{\omega})$ we denote by $\mathcal{B}=\mathcal{B}(V)$ the collection of gaps in $V$, i.e. the collection of \emph{bounded} connected componenets of $\R^2 \backslash \pi(V)$. The following notions were coined by Cornea and Shelukhin in \cite{CorneaShelukhin15}.
\begin{defi}[\cite{CorneaShelukhin15}]
Given a Lagrangian with cylindrical ends $V\subset (\tilde{M},\tilde{\omega})$ we define the \emph{outline of $V$} as the closed subset of $\R^2$
\[
\ou(V):=\pi(V)\cup \left( \cup_{B\in \mathcal{B}}B \right).
\] 
The \emph{shadow of $V$} is then defined as the non-negative number 
\[
\mathcal{S}(V):=\area \left( \ou(V) \right).
\] 
\end{defi}

Denote now by $\mathcal{J}$ (respectively $\tilde{\mathcal{J}}$) the space of smooth almost complex structures on $M$ (respectively $\tilde{M}$) which are compatible with the symplectic structure. We denote by $\tilde{\mathcal{J}}_c \subset \tilde{\mathcal{J}}$ the subset consisting of almost complex structures which are \emph{standard at $\infty$} in the following sense: For every $\tilde{J}\in \tilde{\mathcal{J}}_c$ there exists a compact set $C\subset \R^2$ such that the restriction of $\tilde{J}$ to $(\R^2 \backslash C)\times M$ has the form $\mathfrak{i} \oplus J$, for some $J\in \mathcal{J}$. We say that \emph{$\tilde{J}$ is supported in $C$} and we denote by $\mathcal{\tilde{J}}(C)\subset \mathcal{\tilde{J}}_c$ the subset consisting of almost complex structures which are supported in $C$. Given $\tilde{J}\in \tilde{\mathcal{J}}$ we denote by $A_S(\tilde{M},\tilde{J})$ the minimal symplectic area of a non-constant $\tilde{J}$-holomorphic sphere in $\tilde{M}$. Given a Lagrangian with cylindrical ends $V\subset (\tilde{M},\tilde{\omega})$ we denote by $A_D(\tilde{M},V,\tilde{J})$ the minimal symplectic area of a non-constant $\tilde{J}$-holomorphic disk in $\tilde{M}$ with boundary on $V$. Suppose $\tilde{J}\in \tilde{\mathcal{J}}(C)$ for a compact set $C\subset \R^2$ and let $u$ be a $\tilde{J}$-holomorphic disk/sphere. It follows from the open mapping theorem that if $u$ satisfies $\image (\pi \circ u) \not\subset C$ then $z\mapsto \pi \circ u(z)$ is constant. With this fact at hand it is easy to adapt the usual compactness argument to show that 
\[
A_S(\tilde{M},\tilde{J}), A_D(\tilde{M},V,\tilde{J})>0 \quad \forall \ \tilde{J}\in \tilde{\mathcal{J}}_c.
\]
\begin{defi}
\label{def2}
Let $V\subset (\tilde{M},\tilde{\omega})$ be a Lagrangian with cylindrical ends. We define the bubbling threshold $A(\tilde{M},V)$ of $V$ by
\[
A(\tilde{M},V):=\sup_{\tilde{J}\in \tilde{\mathcal{J}}_c}A(\tilde{M},V,\tilde{J}),
\]
where $A(\tilde{M},V,\tilde{J}):=\min\{A_D(\tilde{M},V,\tilde{J}),A_S(\tilde{M},\tilde{J})\}$. 
\end{defi}
\section{Proofs}
\label{secproof}
We begin by proving our applications of Theorem \ref{adcobthm} and \ref{relChekanov}. The following remark will be used frequently.
\begin{rem}
\label{rem2}
Applying Theorem \ref{relChekanov} with coefficients in a field $\F \neq \Z_2$ requires us to know that the (small) cobordism $V: (L'_i)_{i=1}^{m'} \rightsquigarrow (L_j)_{j=1}^{m}$ is spin. However, in many cases the spin condition follows from the smallness assumption if we know e.g. that every $L_i$ is spin. The idea is the following bootstrapping argument: Suppose the intersection points in (\ref{eq3}) are so few that one can can apply the $\Z_2$-version of Theorem \ref{relChekanov} to verify that the inclusion $i:(\sqcup_iL_i)\hookrightarrow V$ induces injections $0\to H^k(V;\Z_2)\to H^k(\sqcup_iL_i;\Z_2)$ for $k=1,2$. Applying $i^*$ to the Stiefel-Whitney classes we have $i^*(w_k(V))=\sum_{i=1}^mw_k(L_i)=0$ for $k=1,2$. Here we use the assumption that every $L_i$ is spin. It follows that $w_k(V)=0$ for $k=1,2$, so $V$ is spin as claimed.
\end{rem}
\subsection{Proofs of results from Section \ref{app1}}
\begin{proof}[Proof of Theorem \ref{cor2}]
By Theorem \ref{relChekanov} every small elementary Lagrangian cobordism $V:L'\rightsquigarrow L$ satisfies $H_*(V,L;\Z_2)=0=H_*(V,L';\Z_2)$. In particular the inclusions $L,L'\hookrightarrow V$ induce isomorphisms on $\Z_2$-(co)homology. The $\Z_2$-version of the theorem follows. To obtain the $\Z$-version we apply Remark \ref{rem2} to conclude that one of $L$ and $L'$ being spin implies that $V$ is spin. Now we can apply Theorem \ref{relChekanov} to conclude that $H_*(V,L;\F)=0=H_*(V,L';\F)$ for every field $\F$. It therefore follows from the homological universal coefficients theorem \cite[Corollary 3A.6.]{Hatcher02} that $H_*(V,L;\Z)=0=H_*(V,L';\Z)$. Hence, the inclusions $L,L'\hookrightarrow V$ induce isomorphisms on $\Z$-(co)homology.
\end{proof}
\begin{proof}[Proof of Corollary \ref{cor1}]
Given a small Lagrangian cobordism $V:L'\rightsquigarrow L$ we can "bend" its right end in order to obtain a Lagrangian null-cobordism $V':\emptyset \rightsquigarrow (L,L')$. It is not hard to see that this bending can be done in such a way that $V'$ again is small. If $L \pitchfork L'$ then $V'$ has transversally intersecting ends. From the proof of Corollary \ref{cor2} we know that $V$ and hence $V'$ are spin. Now the conclusion follows by applying (\ref{eq3}) to $V'$. To see this, note that $H_*(V',\partial_+V';\F)=H_*(V';\F)\cong H_*(V;\F)\cong H_*(L;\F)$, where the last isomorphism comes from the inclusion $L\hookrightarrow V$.
\end{proof}
\begin{proof}[Proof of Corollary \ref{cor0}]
If $V:\emptyset \rightsquigarrow L$ is an oriented Lagrangian null-cobordism with boundary $L\in \mathcal{L}$ then elementary algebraic topology implies $[L]=0\in H_n(M;\Z)$, so $\chi(L)=0$. For the second part of the corollary, suppose for contradiction that $L\in \mathcal{L}$ admits a \emph{small} Lagrangian null-cobordism $V:\emptyset \rightsquigarrow L$. Then $\partial_+ V=\emptyset$, so $H_0(V,\partial_+ V;\Z_2)=H_0(V;\Z_2)=\Z_2$. On the other hand Theorem \ref{relChekanov} implies that $H_0(V,\partial_+ V;\Z_2)=0$. This contradiction finishes the proof.
\end{proof}
\begin{proof}[Proof of Corollary \ref{whitehead}]
$L$ and $L'$ being simply connected implies that they are both spin. Therefore the assumption that $V$ is small and Theorem \ref{cor2} imply that $H_*(V;L;\Z)=0=H_*(V,L';\Z)$. Now the conclusion follows from Smale's famous $h$-cobordism theorem \cite{Smale62}, \cite[Theorem 9.1]{Milnor65(1)}.
\end{proof}
\subsection{Proofs of results from Section \ref{app2}}
\begin{proof}[Proof of Corollary \ref{cormultend1}]
\label{proofmultend}
Recall that we are considering an $m$-tuple $(L_i)_{i=1}^m\subset \mathcal{L}$, a singleton $L'\in \mathcal{L}$ as well as a small Lagrangian cobordism $V:L' \rightsquigarrow (L_i)_i$. Note that (\ref{eq100}) follows from Theorem \ref{relChekanov} if only we show 
\begin{equation}
\label{eq103}
\dim_{\F}H_n(V,L';\F)\geq m-1
\end{equation}
with $\F=\Z_2$. To see this we consider the diagram
\begin{equation}
\label{eq102}
\xymatrix{
	\cdots  \ar[r] & H_{n+1}(V,\partial V;\F) \ar[r]^{\partial_{\F}} & H_n(\partial V,L';\F)\ar[r]^{i_{\F}} & H_n(V,L';\F) \ar[r] & \cdots \\
 & &  \oplus_{j=1}^mH_n(L_i;\F) \ar[u]_{\cong} \ar[r]& H_n(V;\F), \ar[u]& 
}
\end{equation}
with $\F=\Z_2$. Here the top horizontal line is a piece of the long exact sequence associated with the triple $(V,\partial V, L')$. Note that trivially $\oplus_{i=1}^mH_n(L_i;\Z_2)\cong \Z_2^m$ and $H_{n+1}(V,\partial V;\Z_2)\cong \Z_2$, so $i_{\Z_2}$ induces an embedding $\Z_2^{m-1}\hookrightarrow H_n(V,L';\Z_2)$ which proves (\ref{eq103}). 
\end{proof}
\begin{proof}[Proof of Theorem \ref{cormultend4}]
We first prove the $\Z_2$-version of the result. Recall that we are considering a small Lagrangian cobordism $V:L' \rightsquigarrow (L_i)_{i=1}^m$. Since we are assuming $\mathcal{I}=m-1$ the proof above together with Theorem \ref{relChekanov} and Poincar\'e-Lefschetz duality gives 
\begin{equation}
\label{eq47}
H_k(V,L';\F)=0=H_{n+1-k}(V,\partial_-V;\F) \quad \forall \ k\neq n 
\end{equation}
and
\begin{equation}
\label{eq48}
H_n(V,L';\F)=\F^{m-1}=H_{1}(V,\partial_-V;\F)
\end{equation}
with $\F=\Z_2$. Since the square in (\ref{eq102}) commutes we also know that the map $H_n(V;\Z_2)\to H_n(V,L';\Z_2)$ is onto. Therefore the map $H_k(L';\Z_2)\to H_k(V;\Z_2)$ induced by the inclusion is an isomorphism for all $k< n$. A similar consideration for the long exact sequence associated with the pair $(V,\partial_-V)$ shows that the map $H_k(\sqcup_{i=1}^mL_i;\Z_2)\to H_k(V;\Z_2)$ is an isomorphism for all $k> 0$ and therefore 
\begin{equation}
\label{eq50}
\oplus_{i=1}^m H_k(L_i;\F)\cong H_k(\partial_-V;\F)\cong H_k(V;\F )\cong H_k(L;\F), \quad 0<k< n
\end{equation}
follows for $\F=\Z_2$. This finishes the proof for $\Z_2$-coefficients. For the $\Z$-version we first claim that the assumption that every $L_i$ is spin implies that $V$ is spin. To see this we first check that $V$ is orientable. For this, note that the image of the first Stiefel-Whitney class $w_1(V)\in H^1(V;\Z_2)$ in $H^1(\partial_-V;\Z_2)$ vanishes, so $w_1(V)$ lifts to an element $\alpha \in H^1(V,\partial_-V;\Z_2)$ which is dual to $\alpha \smallfrown [V] \in H_n(V,L';\Z_2)$. Exactness in (\ref{eq102}) and surjectivity of $i_{\Z_2}$ implies that $j_{\Z_2}:H_n(V,L';\Z_2)\to H_n(V,\partial V;\Z_2)$ vanishes. We therefore conclude that
\[
0=j_{\Z_2}(\alpha \smallfrown [V])=w_1(V)\smallfrown [V],
\] 
which by Poincar\'e-Lefschetz duality implies $w_1(V)=0$, so $V$ is orientable. Note that $L$ too is orientable, being a boundary component of an orientable manifold. To see that also $w_2(V)=0\in H^2(V;\Z_2)$ it suffices to note that (\ref{eq47}) implies that $H^2(V,\partial_-V;\Z_2)=0$, so the inlcusion $\partial_-V \hookrightarrow V$ induces an injection $0\to H^2(V;\Z_2)\to H^2(\partial_-V;\Z_2)$. Hence, by Remark \ref{rem2} $w_2(V)=0$ and $V$ is spin. We can therefore fix any field $\F$ and apply Theorem \ref{relChekanov} to conclude $\dim_{\F}H_*(V,L';\F)\leq m-1$. Since every $L_i$ is orientable the isomorphism in (\ref{eq102}) implies $H_n(\partial V,L';\F)\cong \F^m$. Exactly as above we can apply Poincar\'e-Lefschetz duality to conclude (\ref{eq47}) and (\ref{eq48}), this time with coefficients in $\F$. Since these considerations hold for every field $\F$ it follows from the homological universal coefficients theorem \cite[Corollary 3A.6.]{Hatcher02} that (\ref{eq47}) and (\ref{eq48}) also hold with $\F =\Z$. To finish the proof we need to check that $i_{\Z}$ is onto. Since all groups displayed in the top horizontal line of (\ref{eq102}) with $\F=\Z$ are free, we can count ranks to conclude that if $i_{\Z}$ were not onto then $\Coker(i_{\Z})$ would be torsion (see e.g. \cite[Chapter II, Theorem 1.6]{Hungerford80}). Thus, by exactness in (\ref{eq102}) we conclude that, if $i_{\Z}$ were not onto, then $H_n(V,\partial V;\Z)$ would \emph{not} be free. However, duality implies that $H_n(V,\partial V;\Z)\cong H^1(V;\Z)$ is free and therefore $i_{\Z}$ must be onto. As in the previous proof it follows that $H_n(V;\Z)\to H_n(V,L';\Z)$ is onto so that the inclusion $L'\hookrightarrow V$ induces an isomorphism $H_k(L';\Z)\stackrel{\cong}{\longrightarrow}H_k(V;\Z)$ for all $k< n$. A similar consideration shows that the map $H_1(V,\partial_-V;\Z)\to H_0(\partial_-V;\Z)$ is injective. Comparing this to (\ref{eq47}) and (\ref{eq48}) with $\F=\Z$ one sees that the inclusion $\partial_-V \hookrightarrow V$ induces an isomorphism $H_k(\partial_-V;\Z)\stackrel{\cong}{\longrightarrow} H_k(V;\Z)$ for every $k>0$, so (\ref{eq50}) follows with $\F=\Z$. This proves the $\Z$-version of the theorem.
\end{proof}
\begin{proof}[Proof of Corollary \ref{cormultend5}]
Consider a small Lagrangian cobordism $V:L' \rightsquigarrow (L_i)_{i=1}^m$, where $(L_i)_{i=1}^m\subset \mathcal{L}$ is a simple $m$-tuple. If the intersection graph of $(L_i)_{i}$ is a tree then $\mathcal{I}=m-1$, so the homologies of $L'$ and $V$ are computed in Theorem \ref{cormultend4}. Moreover, any equipment of the immersed Lagrangian $(\cup_iL_i)\subset (M,\omega)$ will result in a surgery $\widetilde{\#}_iL_i$ which equals the connected sum $\#_iL_i$ at the level of smooth manifolds, and whose associated Lagrangian "trace of surgery" cobordism $\widetilde{V}:\widetilde{\#}_iL_i \rightsquigarrow (L_i)_{i=1}^m$ has the homotopy type of the subset $(\cup_iL_i)\subset (M,\omega)$. The result is now an easy computation. 
\end{proof}
\begin{proof}[Proof of Corollary \ref{cor6}]
Let $L,L'\in \mathcal{L}$ be monotone and spin and suppose $V:L' \rightsquigarrow L$ is a monotone Lagrangian cobordism. This was the setting in which Chekanov \cite{Chekanov97} originally proved Corollary \ref{cor6}. Chekanov's idea was that the signed count of holomorphic disks in a given class $\alpha \in H_2(M,L;\Z)$ should coincide with signed count of holomorphic disks in $\tilde{M}$ with boundary on $V$ representing the class $i_*(\alpha)\in H_2(\tilde{M},V;\Z)$, where $i:(M,L)\hookrightarrow (\tilde{M},V)$ denotes the inclusion. However, it appears that the proof of this presented in \cite{Chekanov97} contains a gap, because it seemingly requires that $i_*$ is injective. But $i_*$ is injective if $V$ is small! Hence, assume $V:L' \rightsquigarrow L$ is a small Lagrangian cobordism. By Theorem \ref{cor2} we then know that the inclusions $L,L'\hookrightarrow M$ induce isomorphisms $H_*(M,L;\Z)\cong H_*(\tilde{M},V;\Z)\cong H_*(M,L';\Z)$, so $V$ is also monotone and spin. Since the Maslov index is a characteristic class we have bijections 
\begin{equation}
\label{eq02}
\mathcal{D}_L \leftrightarrow \mathcal{D}_V \leftrightarrow \mathcal{D}_{L'}.
\end{equation}
Choose $R>0$ such that $V$ is cylindrical outside the "box" $B:=[\epsilon-R,R-\epsilon]^2$ for some small $\epsilon>0$. As in Definition \ref{defcob1} we will view $V$ as a subset of $([-R,R]\times \R \times M, \tilde{\omega})$ by "cutting off" its ends. For $\tilde{J}\in \tilde{\mathcal{J}}(B)$ and $\alpha_V \in \D_V$ we now consider the moduli space
\[
\widetilde{\mathcal{M}}_V(\alpha_V, \tilde{J}):=\{ u:(\mathbb{D}, \partial \mathbb{D})\to (\tilde{M},V)\ |\ \overline{\partial}_{\tilde{J}}u=0 \ \& \ [u]=\alpha_V \}.
\]
Recall that any $\tilde{J}\in \tilde{\mathcal{J}}(B)$ satisfies $\tilde{J}|_{\R^2 \backslash B}=\mathfrak{i}\oplus J$ for some $J\in \mathcal{J}(M,\omega)$. Therefore, by the open mapping theorem from complex analysis the image of every $\tilde{J}$-holomorphic disk $u\in \widetilde{\mathcal{M}}_V(\alpha_V,\tilde{J})$ passing through a point $(x,y,q)\in V\subset \tilde{M}$ with $|x|>R-\epsilon$ is contained in the fiber $\{(x,y)\}\times M$. It follows that, after perhaps perturbing $\tilde{J}\in \tilde{\mathcal{J}}(B)$ slightly, the space $\widetilde{\mathcal{M}}_V(\alpha_V, \tilde{J})$ (if $\neq \emptyset$) is a $(n+3)$-dimensional manifold with boundary $\partial \widetilde{\mathcal{M}}_V(\alpha_V, \tilde{J})=\widetilde{\mathcal{M}}_L(\alpha_L, J) \sqcup \widetilde{\mathcal{M}}_{L'}(\alpha_{L'} , J)$, where $\alpha_L \in \mathcal{D}_L$ and $\alpha_{L'} \in \mathcal{D}_{L'}$ are the unique elements which correspond to $\alpha_V$ under (\ref{eq02}). This transversality argument is quite straightforward using the well-known fact that constant index 0 disks are automatically transverse (see also \cite{BiranCornea13} for details). The action of the group of automorphisms of $\mathbb{D}$ which preserve $1\in \partial \mathbb{D}$ respects the boundary of $\widetilde{\mathcal{M}}_V(\alpha_V, \tilde{J})$, so the quotient $\mathcal{M}_V(\alpha_V, \tilde{J})$ is a smooth $(n+1)$-dimensional manifold with boundary $\partial \mathcal{M}_V(\alpha_V, \tilde{J})=\mathcal{M}_L(\alpha_L, J) \sqcup \mathcal{M}_{L'}(\alpha_{L'} , J)$. Choosing a spin structure on $V$, we obtain an orientation of $\mathcal{M}_V(\alpha_V, \tilde{J})$ and by Gromov compactness $\mathcal{M}_V(\alpha_V, \tilde{J})$ is compact. The claim of the corollary now follows from commutativity of the diagram
\begin{equation*}
\xymatrix{
	H_{n+1}(\mathcal{M}_V,\partial \mathcal{M}_V;\Z) \ar^{\partial}[d] \ar^{ev}[r] & H_{n+1}(V,\partial V;\Z) \ar^{\partial}[d] \\
 H_{n}(\partial \mathcal{M}_V;\Z) \ar^{ev}[r] & H_{n}(\partial V;\Z) 
}
\end{equation*}
where $\mathcal{M}_V=\mathcal{M}_V(\alpha_V, \tilde{J})$ and $ev$ denotes the evaluation map $ev([u])=u(1)$.
\end{proof}

\begin{proof}[Proof of Corollary \ref{cor00}]
Recall the setting of Corollary \ref{cor00}: We consider $T=T^n_k(a),T'=T^n_{k'}(a)\in \mathcal{L}(\R^{2n},\omega_{\R^{2n}})$ for some $k$ even, $k'$ odd and $n\geq 3$ odd, perturbed in such a way that $T\pitchfork T'$ and $\#(T\cap T')=2$. Suppose $L\in \mathcal{L}(\R^{2n},\omega_{\R^{2n}})$ is a Lagrangian for which we have a \emph{small} Lagrangian cobordism $V:L \rightsquigarrow (T,T')$. Fix $R>0$ such that $V$ is cylindrical outside $B:=[\epsilon -R,R-\epsilon]$ for some small $\epsilon>0$. We will view $V\subset [-R,R]\times \R \times M$. For $\alpha \in \mathcal{D}_V$ and $\tilde{J}\in \tilde{\mathcal{J}}(B)$ we consider the moduli space $\widetilde{\mathcal{M}}^*_V(\alpha;\tilde{J})$ of \emph{simple} $\tilde{J}$-holomorphic disks in $\tilde{M}$ with boundary on $V$ representing the class $\alpha$ (see \cite[Definition 3.1.1.]{BiranCornea07}). For generic $\tilde{J}\in \tilde{\mathcal{J}}(B)$ the quotient $\mathcal{M}^*_V(\alpha;\tilde{J})$ of $\widetilde{\mathcal{M}}^*_V(\alpha;\tilde{J})$ by the group of automorphisms of $\mathbb{D}$ which preserve $1\in \partial \mathbb{D}$ is a smooth $(n+1)$-dimensional manifold with boundary
\begin{equation}
\label{eq000}
\bigcup_{\beta} \mathcal{M}^*_L(\beta;J) \sqcup \bigcup_{\xi} \mathcal{M}^*_{T}(\xi;J) \sqcup \bigcup_{\zeta} \mathcal{M}^*_{T'}(\zeta;J),
\end{equation}
where the unions run over all elements of $\mathcal{D}_L,\mathcal{D}_{T},\mathcal{D}_{T'}$ which hit $\alpha$ when pushed into $H_2(\R^{2(n+1)},V;\Z)$. By \cite[Lemma 2.1]{Chekanov97} there is an odd number of elements $\zeta \in \D_{T'}$ such that $\eta_{T'}(\zeta; \Z_2)=1\in \Z_2$ and an even number of elements $\xi \in \D_{T}$ such that $\eta_T(\xi;\Z_2)=1\in \Z_2$. It follows that by perhaps changing $\alpha \in \D_V$ we can arrange that the parity of the number of $\zeta \in \D_{T'}$ occuring in (\ref{eq000}) for which $\eta_{T'}(\zeta;\Z_2)=1$ differs from the parity of the number of $\xi \in \D_{T}$ occuring in (\ref{eq000}) for which $\eta_{T}(\xi;\Z_2)=1$. Given such a choice of $\alpha$ we choose points $q\in T$ and $q'\in T'$ which are regular for the evaluation maps
\[
\mathcal{M}^*_{T}(\xi;J)\to T \quad \& \quad \mathcal{M}^*_{T'}(\zeta;J)\to T',
\]
associated to all $\zeta \in \mathcal{D}_{T'}$ and all $\xi \in \mathcal{D}_{T}$ occuring in (\ref{eq000}). Now choose a smooth embedding $\gamma: \R \to V$ such that\footnote{Here we view $V$ as a Lagrangian with cylindrical ends.} 
\[
\gamma(t)=\left\{  
\begin{array}{ll}
(t,1,q) & \text{if}\ t\leq \epsilon -R \\
(t,2,q') & \text{if}\ t\geq R-\epsilon 
\end{array}
\right.
\]
After perhaps perturbing $\gamma|_{[\epsilon-R,R-\epsilon]}$ we obtain that the evaluation map $ev:\mathcal{M}^*_V\to V$ as well as its restriction to the boundary $ev|_{\partial \mathcal{M}_V^*}:\partial \mathcal{M}_V^*\to \partial V$ are transverse to $\gamma$. It follows that $N:=ev^{-1}(\gamma)\subset \mathcal{M}^*_V$ is a smooth 1-dimensional manifold with boundary $\partial N=N\cap \partial \mathcal{M}_V^*$. The choice of $\alpha$ implies that $N$ has an odd number of boundary points. Hence, we can find a sequence $(u_j)_{j\in \N}\subset N$ which has no convergent subsequence in $N$. Applying Gromov convergence to $(u_j)$ we obtain a (Gromov) convergent subsequence, again denoted by $(u_j)$. By construction there are two possible (Gromov) limits:
\begin{enumerate}[a)]
\item
$(u_j)$ converges to a genuine cusp curve. I.e. the limit consists of multiple ($\# >1$) $\tilde{J}$-holomorphic disks in $\tilde{M}$ with boundary on $V$.
\item
$(u_j)$ converges to a $\tilde{J}$-holomorphic disk $u:(\mathbb{D},\partial \mathbb{D})\to (\tilde{M},V)$ representing the class $\alpha \in \D_V$. Since $ev(u)\in \gamma$ we conclude that $u$ cannot be simple.
\end{enumerate} 
Suppose now for contradiction that $V$ is orientable. Under this assumption we study the limits in the two cases a) and b) above: In case a) one of the limit disks $v:(\mathbb{D},\partial \mathbb{D})\to (\tilde{M},V)$ will satisfy $0<\tilde{\omega}(v)<a$ and $\mu_V(v)\leq 1$. Moreover, since $V$ is assumed orientable we must have $\mu_V(v)$ even, so 
\begin{equation}
\label{eq0000}
0<\tilde{\omega}(v)<a \quad \& \quad \mu_V(v)\leq 0.
\end{equation}
It follows that $[\partial v]\in H_1(V;\Z)/H_1(\partial_- V;\Z)$ generates an infinite cyclic subgroup, so the image $\beta \in H_1(V,\partial_- V;\Z)$ of $\partial v$ also generates an infinite cyclic subgroup. If instead we are in case b) we apply a result due to Lazzarini \cite{Lazzarini00} to extract a $\tilde{J}$-holomorphic disk $v:(\mathbb{D}, \partial \mathbb{D})\to (\tilde{M},V)$ which (by the same considerations as above) must satisfy (\ref{eq0000}). Again we conclude that the image $\beta \in H_1(V,\partial_- V;\Z)$ of $\partial v$ generates an infinite cyclic subgroup. We conclude that in both cases a) and b) $\Range (H_1(V;\Z)\to H_1(V,\partial_- V;\Z))$ contains an infinity cyclic subgroup. However, by studying the long exact sequence in homology one sees that $H_1(V,\partial_- V;\Z)/H_1(V;\Z)$ is an infinity cyclic subgroup, so we conclude that $\ran(H_1(V,\partial_- V;\Z))\geq 2$ which implies 
\[
\dim_{\Z_2}H_1(V,\partial_- V;\Z_2)\geq 2.
\]
Since $V$ is small Theorem \ref{relChekanov} now implies 
\begin{equation}
\label{eq001}
H_l(V,\partial_-V;\Z_2)\cong \left\{
\begin{array}{ll}
\Z_2\oplus \Z_2 & \text{if}\ l=1 \\
0 & \text{if}\ l\neq 1.
\end{array}
\right.
\end{equation}
Hence, $\chi(V,\partial_-V)=-2$ which (by Theorem \ref{adcobthm}) contradicts the assumption that $V$ is orientable. This contradiction shows that $V$ must be non-orientable. With this at hand one again easily computes (using the first Stiefel-Whitney class) that (\ref{eq001}) must hold. By Poincar\'e-Lefschetz duality this implies $H^1(V,L;\Z_2)\cong H_1(V,L;\Z_2)=0$ so that the first Stiefel-Whitney class of $L$ is non-trivial. We conclude that $L$ too is non-orientable. Computing the $\Z_2$-homology of $L$ and $V$ now comes down to writing out long exact sequences and using (\ref{eq001}). This finishes the proof of Corollary \ref{cor00}.
\end{proof}

\begin{proof}[Proof of Proposition \ref{propkazvas}]
The statement is a consequence of the following basic fact: A non-trivial characteristic class $0\neq \alpha_{L'} \in H^l(L';\Z_2)$ in degree $l\geq 1$ which satisfies $\alpha_{L_i}=0\in H^l(L_i;\Z_2)$ for every $i=1,\ldots, m$ gives rise to one dimension in $H^l(V,\partial_-V;\Z_2)$: $0\neq \alpha_V \in H^l(V;\Z_2)$ is in the image of the map $H^l(V,\partial_- V ;\Z_2)\to H^l(V;\Z_2)$. Clearly the image of $H^0(\partial_- V;\Z_2)\to H^1(V,\partial_- V;\Z_2)$ has dimension $m-1$. Hence, the bound in the proposition follows from Theorem \ref{relChekanov}.
\end{proof}


\subsection{Proofs of Theorem \ref{adcobthm} and \ref{relChekanov}}
The proof of Theorem \ref{adcobthm} is an easy consequence of Weinstein's Lagrangian tubular neighborhood theorem, so we will carry it out first.
\subsubsection{Proof of Theorem \ref{adcobthm}}
Consider two ordered, oriented tuples $(L_i)_{i=1}^m,(L'_j)_{j=1}^{m'}\subset \mathcal{L}$ as well as an \emph{oriented} Lagrangian cobordism $V: (L'_j)_j\rightsquigarrow (L_i)_i$. The non-oriented case follows the same line of ideas and will therefore not be mentioned further.
By Remark \ref{rem1} it suffices to compute $\chi(V,\partial_{+}V)$. To do this we may as well assume that $L_i\pitchfork L_j$ and $L'_i\pitchfork L'_j$ for $i\neq j$. This can be achieved by attaching small Lagrangian suspensions to the ends of $V$ \cite[Section 3.1 E]{Polterovich01}, which clearly does not change the topology of V. Now fix a Darboux-Weinstein neighborhood $U\subset \tilde{M}$ of $V$, so that we have a neighborhood $W\subset T^*V$ of the zero-section $V\subset T^*V$ as well as a symplectic identification $U\approx W$ which restricts to the identity on $V$. Denote by $B:=[-R,R]^2\subset \R^2$ a "box" such that $V$ is cylindrical outside $B\times M$. We may choose $U$ such that it is of product type outside $B\times M$. Denote by $g_{\R^2}$ the standard Euclidean metric on $\R^2$ and fix a Riemannian metric $g_M$ on $M$. Now fix a Morse function $f\in C^{\infty}(V)$ such that $-\nabla^g f$ points \emph{outwards} along $\partial_+V$ and \emph{inwards} along $\partial_-V$, where $g:=g_{\R^2}\oplus g_{M}$.
We further require that 
\begin{equation}
\label{eq01}
f(x,y,p)=\left\{
\begin{array}{ll}
\sigma_j^-(x), & \forall \ (x,y,p)\in (-\infty,-R)\times \{j\} \times L_j \\
\sigma_j^+(x), & \forall \ (x,y,p)\in (R,\infty) \times \{j\} \times L_j'
\end{array}
\right.
\end{equation}
where $\sigma_j^-$ and $\sigma_j^+$ have the form $\sigma_j^{\pm}(x)=\alpha x+\beta^{\pm}$ for a constant $\alpha<0$ and constants $\beta^{\pm}\in \R$. We extend $f$ to a (non-compactly supported and autonomous) Hamiltonian $F\in C^{\infty}(\tilde{M})$ by first extending it constantly along fibers in $W\approx U$ and then cutting off outside of a fiberwise convex neighborhood containing both $\Graph (df)$ and $\Graph(-df)$. If $f|_{V\cap (B\times M)}$ is $C^2$-small then we can achieve that $V_1:=\phi_{F}^1(V)$ is cylindrical outside $B\times M$ (see Figure \ref{fig2}). We assume that $f$ is chosen such that this is the case. We equip $V_1$ with the orientation induced by the diffeomorphism $\phi_F^1:V\to V_1$. We now have an identification $\Crit(f)\approx V\cap V_1$ and it is easy to check that
\[
(-1)^{|q|_f}=(-1)^{\frac{(n+1)n}{2}}I_q(V,V_1) \quad \forall \ q\in \Crit(f)\approx (V\cap V_1),
\]  
where $I_q(V,V_1)$ denotes the intersection index at $q$ with respect to the orientation $\tilde{\omega}^{n+1}$ of $\tilde{M}$ and $|q|_f$ denotes the Morse index. Since the Morse homology of $f$ is $H_*(V,\partial_+ V)$ we conclude that 
\begin{equation}
\label{eq2}
\chi(V,\partial_+ V)=\sum_{q\in \Crit(f)}(-1)^{|q|_f}=(-1)^{\frac{(n+1)n}{2}}I(V,V_1),
\end{equation}
where $I(V,V_1)$ denotes the intersection index of $(V,V_1)$ in $\tilde{M}$. Standard arguments in differential topology \cite{GuilleminPollack10} imply that if $\{\rho_t\}_{t\in [0,1]}$ is an isotopy of $\tilde{M}$ supported in $I\times \R \times M$ for some compact interval $I$ such that $V\pitchfork \rho_1(V_1)$ then $I(V,V_1)=I(V,\rho_1(V_1))$. Denote by $\eta:[-R-2,R+2]\to \R$ a compactly supported bump-function such that
\[
\eta'(t)\left\{  
\begin{array}{ll}
\geq 0 & \text{if}\ t\in [-R-2,-R-1] \\
\leq 0 & \text{if}\ t\in [R+1,R+2]
\end{array}
\right.
\] 
and $\eta=C$ on $[-R,R]$ for some large constant $C>0$. Denote by $\{\rho_t\}_{t\in [0,1]}$ the isotopy generated by $-\eta\partial_y$. If $C$ is large enough it is easy to see that $V\pitchfork \rho_1(V_1)$ and each intersection point $q=(x,y,p)\in V\cap \rho_1(V_1)\subset \R^2 \times M$ corresponds to some $p\in L_i\cap L_j$ for $i< j$ if $x<0$ or some $p\in L'_i\cap L'_j$ for $i< j$ if $x>0$. It suffices to compare $I_q(V,\rho_1(V_1))$ to $I_p(L_i,L_j)$ in the former case and to $I_p(L'_i,L'_j)$ in the latter case. Recall that we are using the convention that the orientation $\partial V$ inherits as a boundary of $V$ corresponds to the orientation of the $L_i$ and $L'_j$ via the convention $\partial V=(-\sqcup_iL_i) \sqcup (\sqcup_jL'_j)$. One therefore easily checks that if $q=(x,y,p)\in V\cap \rho_1(V_1)$ with $x<0$ and $p\in L_i\cap L_j$ for $i<j$ then
\[
I_q(V,\rho_1(V_1))=(-1)^{n+1}I_p(L_i,L_j).
\]
If on the other hand $q=(x,y,p)\in V\cap \rho_1(V_1)$ with $x>0$ and $p\in L'_i\cap L'_j$ for $i<j$ then
\[
I_q(V,\rho_1(V_1))=(-1)^n I_p(L'_i,L'_j)=-(-1)^{n+1} I_p(L'_i,L'_j).
\]
It follows that
\[
I(V,\rho_1(V_1))=
(-1)^{n+1} \left( \sum_{1\leq i <j\leq m}I(L_i,L_j)- \sum_{1\leq i <j\leq m'}I(L'_i,L'_j) \right)
\]
Together with (\ref{eq2}) this finishes the computation of $\chi(V,\partial_+ V)$ and therefore the proof of Theorem \ref{adcobthm}.

\subsubsection{Proof of Theorem \ref{relChekanov}}
From now on we consider a \emph{small} Lagrangian cobordism $V: (L'_1,\ldots , L'_{m'})\rightsquigarrow (L_1,\ldots , L_{m})$ which is spin. The $\Z_2$-case when $V$ is not assumed spin follows the same line of arguments and will not be mentioned further. Fix once and for all a $\tilde{J}\in \tilde{\mathcal{J}}_c$ such that $\mathcal{S}(V)<A(\tilde{M},V,\tilde{J})$ together with a small $\delta >0$ such that
\begin{equation}
\label{eq10}
\Delta:= \mathcal{S}(V)+2\delta <A(\tilde{M},V,\tilde{J}).
\end{equation} 
Fix also $R>0$ such that $\tilde{J}$ is supported in $[-R,R]^2\times M$ and set $B:=[-R,R]^2$.

\subsubsection{Shaping $V$}
\label{secshape}
We will reduce to the situation where $V$ is a Lagrangian with cylindrical ends satisfying the following conditions: There exists a $\beta \in C^{\infty}_c(\R; [0,\infty))$ such that  
\begin{equation}
\label{equ1}
\int_{-\infty}^{\infty}\beta(s) ds< \mathcal{S}(V)+\delta,
\end{equation}
and such that $V$ is cylindrical outside the set $Y\times M$ where 
\[
Y:=\{(x,y)\in \R^2\ :\ 0< y< \beta(x)  \}.
\]
I.e. we can write
\begin{equation}
\label{eq11}
V\cap (\R^2\backslash Y\times M)=\left( \bigsqcup_{j=1}^m I^-_j \times \{a_j^-\}\times L_j \right) \cup \left( \bigsqcup_{j=1}^{m'} I^+_j \times \{a_j^+\}\times L'_j \right)
\end{equation}
for intervals of the type $I_j^{-}=(-\infty, r^-_j]$ and $I_j^{+}=[r^+_j,\infty)$ and numbers $a_j^{\pm}\in \R$ satisfying $a^{\pm}_j\neq a^{\pm}_i$ for $i\neq j$. By perhaps increasing $R$ we may as well assume that $Y\subset B$.
\begin{figure}[htbp]
	\centering
	\begin{tikzpicture}[scale=1.8]
\draw [thick, ->] (-3,0) -- (3,0);
\node [right] at (3,0) {\small $\R$};
\node [below] at (2.6,0) {\small $R$};
\node [below] at (-2.6,0) {\small $-R$};

\draw [thick, fill=black] plot [smooth cycle, tension=1] coordinates {(0,0.1) (0.4,0.2) (0.8,0.1) (1.2,0.2) (1.5,0.3) (1.5,0.7) (1.3,1) (1,0.9) (0.5,1) (0,0.8) (-0.5,1) (-1,0.7) (-1.5,0.5) (-1.5,0.25) (-1,0.2)};
\draw [thick] (1.5,0.3) -- (3,0.3);
\draw [thick] (1.5,0.7) -- (3,0.7);
\draw [thick] (1.3,1) -- (3,1);
\draw [thick] (-3,0.5) -- (-1.5,0.5);
\draw [thick] (-3,0.25) -- (-1.5,0.25);
\node at (0,0.5) {\small  $\color{white} \pi(V)$};

\draw [thick, red] plot [smooth cycle, tension=0.5] coordinates {(0,0.2) (0.4,0.1) (0.7,0.2) (1.1,0.1) (1.3,0.2) (1.5,0.55) (1.5,0.85) (0.9,0.9) (0.5,1) (0,0.7) (-0.5,0.9) (-1,0.8) (-1.4,0.35) (-1.5,0.1) (-1,0.1)};
\draw [thick, red] (1.3,0.2) -- (3,0.2);
\draw [thick, red] (1.5,0.55) -- (3,0.55);
\draw [thick, red] (1.5,0.85) -- (3,0.85);
\draw [thick, red] (-3,0.35) -- (-1.4,0.35);
\draw [thick, red] (-3,0.1) -- (-1.5,0.1);

\draw [cyan, thick] (-1.9,0) to  [out=0, in=250]  (-1.6,0.5) to [out=70, in=200] (-1,0.9) to [out=20, in=130] (-0.1,0.9) to [out=310, in=180] (0.5,1.1) to [out=0, in=180] (1,0.95) to [out=0, in=180] (1.3,1.1) to [out=0, in=100] (1.65,0.8) to [out=280, in=180] (1.9,0);
\node [above] at (0.5,1.05) {\small \color{cyan} $\Graph(\beta)$};

\end{tikzpicture}
	\caption{The geometric picture here indicates the situation in the proof of both Theorem \ref{adcobthm} and \ref{relChekanov}, $\Graph(\beta)$ is only needed for the proof of the latter. The projection of $V$ to $\R^2$ is indicated in black and the graph of $\beta$ is indicated in blue. The red figure outlines the projection of $\phi_{F}^1(V)$ to $\R^2$.}
	\label{fig2}
\end{figure}
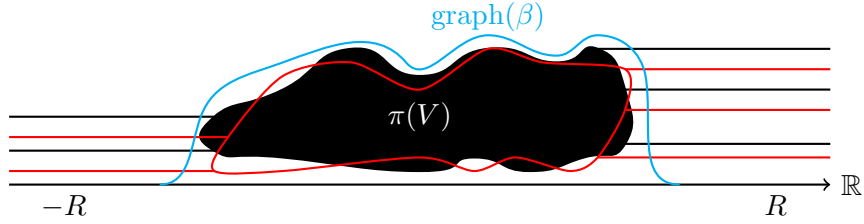
\begin{rem}
	The assumption that $\beta$ satisfying (\ref{equ1}) exists and (\ref{eq11}) is satisfied can be made without loss of generality: If $V$ does not satisfy these conditions then it is easy to find $\psi \in \Symp(\R^2,\omega_{\R^2})$ such that the Lagrangian $\tilde{V}:=\psi \times \id_M(V)\subset (\tilde{M},\tilde{\omega})$ satisfies them. All structures used in the proof below can then be conjugated by $\psi \times \id_M$ in order to transfer the results from $\tilde{V}$ to $V$. From now on we will therefore (without further mentioning) assume the existence of $\beta$ such that (\ref{equ1}) and (\ref{eq11}) are satisfied. 
\end{rem}

\subsubsection{Computing $H_*(V,\partial_{\pm}V; \F)$}
\label{subseqhom}
We will estimate the dimension of $H_*(V,\partial_+V; \F)$ using (a suitable adaption of) local Floer homology. To do this we fix a Morse function on $V$ as in the proof of Theorem \ref{adcobthm}. However, this time we need to be a bit more specific about the choice we make. More precisely, instead of (\ref{eq01}) $f$ is this time required to staisfy 
\begin{equation*}
f(x,y,p)=\left\{
\begin{array}{ll}
\sigma_j^-(x), & \forall \ (x,y,p)\in I_j^-\times \{a^-_j\} \times L_j \\
\sigma_j^+(x), & \forall \ (x,y,p)\in I_j^+\times \{a^+_j\} \times L_j',
\end{array}
\right.
\end{equation*}
where again $\sigma_j^-$ and $\sigma_j^+$ have the form $\sigma_j^{\pm}(x)=\alpha x+\beta^{\pm}$ for a constant $\alpha<0$ and constants $\beta^{\pm}\in \R$. Here $I^{\pm}_j$ denote the intervals from (\ref{eq11}). We extend $f$ exactly like before to a (non-compactly supported and autonomous) Hamiltonian $F\in C^{\infty}(\tilde{M})$, and by choosing $f$ such that $f|_{V\cap (Y\times M)}$ is $C^2$-small we achieve that $\phi_{F}^1(V)$ is cylindrical outside $Y\times M$. Note that $\phi^t_{F}(V) \pitchfork V$ for all $t \in [-1,1]\backslash \{0\}$. Consider now the strip $Z:=\R \times [0,1]$ with coordinates $z=(s,t)$. We study solutions $u\in C^{\infty}(Z,\tilde{M})$ to the problem
\begin{equation}
\label{eq1}
\left\{
\begin{array}{l}
\partial_su +\tilde{J}'_t(u)(\partial_tu-X_F(u))=0 \\
u(\R \times \{0,1\})\subset V
\end{array}
\right.
\end{equation} 
where $\tilde{J}'=\{\tilde{J}'_t\}_{t\in [0,1]}\subset \tilde{\mathcal{J}}$ denotes a smooth path of almost complex structures and $X_F$ denotes the symplectic gradient of $F$, defined by $i_{X_F}\tilde{\omega}=-dF$. Given a solution $u\in C^{\infty}(Z,\tilde{M})$ to (\ref{eq1}) we recall that its energy (with respect to $\tilde{J}'$) is defined by 
\[
E_{\tilde{J}'}(u):=\int_{-\infty}^{\infty}\int_0^1 \tilde{\omega}(\partial_su,\tilde{J}'_t(u)\partial_su)dtds.
\] 
Due to our non-compact setting we will need to impose some restrictions on $\tilde{J}'$ in order to obtain a well-defined Floer theory. In order to do so we first introduce a bit of notation. Denote by $K\subset \R^2$ a compact subset such that $Y\subset K$. We will then denote by $\tilde{\mathcal{J}}_F(K)$ the space of smooth paths of $\tilde{\omega}$-compatible almost complex structures $\tilde{J}'=\{\tilde{J}'_{t} \}_{t\in [0,1]}$ satisfying the condition that 
\begin{equation}
\label{eq24}
(\phi_F^t)^*\tilde{J}'_t|_{(\R^2 \backslash K)\times M}=(\mathfrak{i} \oplus J'_t )|_{(\R^2 \backslash K)\times M}
\end{equation}
for all $t\in [0,1]$, where $\{J'_{t}\}_{t\in [0,1]}$ is some smooth path of $\omega$-compatible almost complex structures on $M$. 
\begin{rem}
It was shown in \cite{BiranCornea14} and \cite{Bisgaard16} that a generic path in $\tilde{\mathcal{J}}_F(K)$ is regular for (\ref{eq1}) in the usual sense of Floer theory. By this we mean both that transversality is achieved for generic $\tilde{J}'\in \tilde{\mathcal{J}}_F(K)$ and that there is a compact subset of $\tilde{M}$ (depending on $F$ and $K$) which contains the image of every finite energy solutions to (\ref{eq1}) for any path $\{\tilde{J}'_t\}_{t\in [0,1]} \in \tilde{\mathcal{J}}_F(K)$. In the following we will need to consider variations of (\ref{eq1}) and therefore also variations of the almost complex structures. The precise equations (and therefore also transversality issues) we will face have been dealt with in practically identical settings before (see for instance \cite{BiranCieliebak02}, \cite{BiranCornea14}, \cite{Oh97} and references therein). The only non-standard aspect here is the compactness issue. However, in each case below compactness follows directly from the arguments in \cite{BiranCornea14}.
\end{rem}
Given a path $\{\tilde{J}'_t\}_{t\in [0,1]} \in \tilde{\mathcal{J}}_F(K)$ and $t\in [0,1]$ we can define $A_S(\tilde{M},\tilde{J}'_t)$ and $A_S(\tilde{M},V,\tilde{J}'_t)$ exactly as in Section \ref{secbubcob}. Due to (\ref{eq24}) it is an easy consequence of the of the open mapping theorem in complex analysis and the usual compactness argument that 
\[
A_S(\tilde{M},\tilde{J}'_t),\ A_S(\tilde{M},V,\tilde{J}'_t)>0.
\]
It follows that 
\[
A(\tilde{M},V,\tilde{J}'_t):=\min\{ A_S(\tilde{M},\tilde{J}'_t), A_S(\tilde{M},V,\tilde{J}'_t) \}>0 \quad \forall \ t\in [0,1].
\]
For the next lemma, note that the path $[0,1]\ni t\mapsto (\phi^t_F)_*\tilde{J}$ is an element of $\tilde{\mathcal{J}}_F(B)$. 
\begin{lemma}
\label{lem1}
If $F|_{B\times M}$ is sufficiently $C^{\infty}$-small then
\begin{enumerate}[a)]
	\item 
	$\Delta <A(\tilde{M},V,(\phi^t_F)_*\tilde{J})$ for all $t\in [0,1]$.
\end{enumerate}
Moreover, if in addition $\{ \tilde{J}'_{t}\}_{t\in [0,1]}\in \tilde{\mathcal{J}}_F(B)$ is sufficiently $C^{\infty}$-close to the path $t\mapsto (\phi^t_F)_*\tilde{J}$ then the following two conditions are satisfied:
\begin{enumerate}[a)]
\setcounter{enumi}{1}
	\item
	Given a solution $u\in C^{\infty}(Z,\tilde{M})$ to (\ref{eq1}) it holds that $E_{\tilde{J}'}(u)\leq \Delta$ if and only if $E_{\tilde{J}'}(u)\leq \delta$. Moreover, if $E_{\tilde{J}'}(u)<\infty $ then $u(Z)\subset U$ if and only if $E_{\tilde{J}'}(u)\leq \delta$.
	\item
	$\Delta <A(\tilde{M},V,\tilde{J}'_t)$ for all $t\in [0,1]$.
\end{enumerate} 
\end{lemma}
\begin{proof}
The proof of a) is a standard compactness argument. Note that for any $\epsilon>0$ the path $[0,1]\ni t\mapsto (\phi^t_{\epsilon F})_*\tilde{J}$ is an element of $\tilde{\mathcal{J}}_{\epsilon F}(B)$. Therefore any $(\phi^t_{\epsilon F})_*\tilde{J}$-holomorphic sphere (respectively disk) into $\tilde{M}$ (respectively $(\tilde{M},V)$) whose $\R^2$-component is non-constant is contained in $B\times M$. Hence we can apply Gromov compactness. Given a), the points b) and c) are nothing but special cases of \cite[Propositions 17.1.2 and 17.1.3]{Oh152}. Note that in \cite{Oh152} the results are stated for closed Lagrangians (see also \cite{Oh96}, \cite{Oh97} and \cite{Chekanov98}). However, since Gromov compactness applies, the proof in our setting is identical to the ones in \cite{Oh152}.
\end{proof}

From now on we assume that the data $\tilde{J}'$ and $F$ is chosen according to Lemma \ref{lem1} and that $J'$ is regular for (\ref{eq1}). We will now discuss the Floer chain complex which we will be using for the proof of Theorem \ref{relChekanov}. Our main reference for Lagrangian Floer homology is Zapolsky's excellent paper \cite{Zapolsky15}, where the orientation issues for Floer homology are worked out in every detail. For more details on the construction of the Floer chain complex we therefore refer to \cite{Zapolsky15} (see also \cite{Seidel08}). We will denote by $\Omega_V$ the space of equivalence classes of pairs $\widetilde{\gamma}=[\gamma,\widehat{\gamma}]$ where $\gamma:([0,1],\{0,1\})\to (\tilde{M},V)$ is a smooth curve and $\widehat{\gamma}$ is a \emph{capping} for $\gamma$. The equivalence relation is given by identifying two cappings if and only if they have equal symplectic areas and equal Maslov indices. Elements $\widetilde{\gamma}=[\gamma,\widehat{\gamma}]\in \Omega_V$ for which $\gamma$ is an integral curve of $X_F$ are exactly the critical points of the action functional $\mathcal{A}_{F:V}:\Omega_V\to \R$, defined by
\[
\mathcal{A}_{F:V}(\widetilde{\gamma}=[\gamma,\widehat{\gamma}])=\int_0^1F(\gamma(t))dt - \int \widehat{\gamma}^*\tilde{\omega}.
\]
We define 
\begin{equation}
\label{eq12}
CF(F:V):=\bigoplus_{\widetilde{\gamma}\in \Crit(\mathcal{A}_{F:V})}C(\widetilde{\gamma}),
\end{equation}
where $C(\widetilde{\gamma})\cong \Z$ is generated by the two orientations of a suitable determinant line bundle of Fredholm operators defined on representatives of $\widehat{\gamma}$ as in \cite{Zapolsky15}. Note that since we identify cappings which have the same symplectic area and Maslov indices, $C(\widetilde{\gamma})$ is only well-defined once we have fixed a spin structure on $V$ so that we can identify the different rank 1 $\Z$-modules coming from different equivalent cappings \cite[Section 7.3]{Zapolsky15}.\footnote{In fact \cite{Zapolsky15} only requires that a relative $\text{Pin}^{\pm}$-structure for $V$ has been chosen. However, for our purposes it is more convenient to assume $V$ is spin, so we will require the choice of a spin structure.} We will therefore fix a spin structure on $V$ from now on. We also define $\Gamma:=\pi_2(\tilde{M},V)/\sim$ where $a\sim b$ if and only if $\tilde{\omega}(a)=\tilde{\omega}(b)$ and $\mu_V(a)=\mu_V(b)$.\footnote{Here $\mu_V$ denotes the Maslov class of $V$.} Then 
\begin{equation}
\label{eq45}
\tilde{\omega}\times \mu_V:\Gamma \to \R \times \Z
\end{equation}
is a monomorphism and $CF_*(F:V)$ is a $\Gamma$-module. In fact, by the construction of $F$, every $\widetilde{\gamma}=[\gamma, \widehat{\gamma}]\in \Crit(\mathcal{A}_{F:V})$ is naturally identified with a pair $[\gamma,\widehat{\gamma}]\approx (q,\widehat{q})$ where $q \in \Crit(f)\subset V$ and $\widehat{q}\in \Gamma$. We denote by $CF_{0}(F:V)\subset CF(F:V)$ the direct sum of the $C(\widetilde{\gamma})$ for which $\widetilde{\gamma}=[\gamma, \widehat{\gamma}]\approx (q,\widehat{q})\in \Crit(\mathcal{A}_{F:V})$ for which $\widehat{q}=0\in \Gamma$. From this point of view it is easy to see that 
\begin{equation}
\label{eq43}
CF_{0}(F:V)\otimes_{\Z}\Gamma \cong CF(F:V)
\end{equation}
as $\Gamma$-modules. A crucial ingredient for understanding Chekanov's construction is the \emph{length} between elements $\widetilde{\gamma}_-,\widetilde{\gamma}_+\in \Crit(\mathcal{A}_{F:V})$, defined by
\[
l(\widetilde{\gamma}_-,\widetilde{\gamma}_+):=\mathcal{A}_{F:V}(\widetilde{\gamma}_-)-\mathcal{A}_{F:V}(\widetilde{\gamma}_+)\in \R.
\]
It is important to note that $l$ is $\Gamma$ bi-invariant. We denote by $\mathcal{M}(F,\tilde{J}',\widetilde{\gamma}_-,\widetilde{\gamma}_+)$ the moduli space of finite-energy and \emph{unparametrized} solutions $u$ of (\ref{eq1}) satisfying the asymptotic conditions 
\[
\lim_{s\to -\infty}u_s=\widetilde{\gamma}_- \quad \& \quad \lim_{s\to \infty}u_s=\widetilde{\gamma}_+
\] 
in $\Omega_V$. For such $u$ we have the energy identity
\[
0\leq E_{\tilde{J}'}(u)= \tilde{\omega}(u)+f(q_-)-f(q_+)=l(\widetilde{\gamma}_-,\widetilde{\gamma}_+),
\]
where we set $\widehat{\gamma}_{\pm}=(q_{\pm},\widehat{q}_{\pm})$. In particular, if $\widehat{q}_{\pm}=0$ we see that $E_{\tilde{J}'}(u)=f(q_-)-f(q_+)$, so if the Hofer norm of $F$ satisfies 
\begin{equation}
\label{eq44}
|\! |F|_{B\times M}|\! |\leq \delta, 
\end{equation}
then automatically $u(Z)\subset U$ by Lemma \ref{lem1}. After perhaps scaling $F$ we can (and will) assume that $F$ has been chosen to satisfy (\ref{eq44}) from now on. We can then define a $\Gamma$-linear operator $\partial:CF(F:V)\to CF(F:V)$ by declaring that its $(\widetilde{\gamma}_-,\widetilde{\gamma}_+)$'th matrix element be 0 if either $l(\widetilde{\gamma}_-,\widetilde{\gamma}_+)>\delta$ or $\dim \mathcal{M}(F,\tilde{J}',\widetilde{\gamma}_-,\widetilde{\gamma}_+) \neq 0$ and
\begin{equation}
\label{eq30}
\sum_{u\in \mathcal{M}(F,\tilde{J}',\widetilde{\gamma}_-,\widetilde{\gamma}_+)}C(u):C(\widetilde{\gamma}_-)\to C(\widetilde{\gamma}_+)
\end{equation}
if $l(\tilde{\gamma}_-,\widetilde{\gamma}_+)\leq \delta$ and $\dim \mathcal{M}(F,\tilde{J}',\widetilde{\gamma}_-,\widetilde{\gamma}_+)=0$. Here $C(\widetilde{\gamma})$ denotes the $\Z$-linear operator defined in \cite[Section 3.8.1.1]{Zapolsky15}. We point out that $\partial$ being $\Gamma$-invariant is a non-trivial matter. This fact uses the choice of a spin structure for $V$ \cite[Section 7.3]{Zapolsky15}. Note that by the remarks above $\partial(CF_0(F:V))\subset CF_0(F:V)$, so we have an operator $\partial|_{CF_0(F:V)}:CF_0(F:V)\to CF_0(F:V)$. From the point of view of the identification 
(\ref{eq43}) we see that
\[
\partial|_{CF_0(F:V)}\otimes_{\Z}\id_{\Gamma}=\partial.
\]
because of $\Gamma$-linearity. We will therefore denote $\partial|_{CF_0(F:V)}$ simply by $\partial$. Similarly, given a field $\F$ we continue to denote the induced operator on $CF(F:V;\F):=CF(F:V)\otimes_{\Z}\F$ by $\partial$.
\begin{prop}
	$(CF_0(F: V;\F),\partial)$ is a chain complex (i.e. $\partial^2=0$) and its homology $HF_0(F:V;\F):=H(CF_0(F: V;\F),\partial)$ satisfies 
	\begin{equation}
	\label{eq104}
	HF_0(F: V;\F)\cong H_*(V,\partial_+ V;\F).
	\end{equation}
\end{prop}
\begin{proof}
	For closed Lagrangian submanifolds this is a classical result for whose proof we refer to \cite[Section 17.2]{Oh152}. The only non-standard aspect when checking $\partial^2=0$ in our situation is making sure that Floer trajectories cannot "escape" along the non-compact ends corresponding to $L_1,...,L_m,L_1',...,L_{m'}'$. This is achieved by simply choosing almost complex structures which are the restriction of paths form $\tilde{\mathcal{J}}_F(B)$ to $U\subset T^*V_0$. One can check this using the arguments from \cite{BiranCornea13}. Checking (\ref{eq104}) can now be done using a PSS argument. This has been carried out in the setting of Lagrangian cobordisms in \cite{BiranCornea13} or \cite{Bisgaard16}. Those accounts easily adapt to our setting.
\end{proof}


\subsubsection{Chekanov's homotopy lemma}
The inequality in Theorem \ref{relChekanov} will follow from an observation due to Chekanov. We will need a slightly modified version of his beautiful result, so we cover the details we need here. Consider a subgroup $A\leq \R \times \Z$ and denote by $\lambda:A\to \R$ the homomorphism given by projection to the first coordinate. Given a field $\F$ we consider the group ring $\Lambda:=\F [A]$. We write an element of $\Lambda$ as a finite sum
\begin{equation}
\label{eq29}
\sum_k f_kT^{a_k}
\end{equation}
where $a_k \in A$ and $f_k \in \F$. We note that $\Lambda$ is both a commutative ring with $1\neq 0$ as well as a $\F$-vector space. Consider also the natural positive and negative $\F$-subspaces
\begin{align*}
	\Lambda_{\pm}:=&\left\{ \sum_i f_iT^{a_i} \ : \ \pm \lambda(a_i)\geq 0 \ \forall \ i \right\}
\end{align*}
together with their $\F$-linear "projections"
\begin{align*}
P_{\pm}:\Lambda &\to \Lambda_{\pm} \\ 
\sum_k f_kT^{a_k} & \mapsto \sum_{k: \pm \lambda(a_k)\geq 0} f_kT^{a_k}
\end{align*} 
Given a finite dimensional $\F$-vector space $W$ we obtain a free $\Lambda$-module $W\! \otimes_{\F}\! \Lambda$ with $\ran_{\Lambda}( W\! \otimes_{\F}\! \Lambda )= \dim_{\F}(W)$. Considering the $\F$-linear subspace $W_0:= W\otimes_{\F}\F[\ker \lambda] \leq W\! \otimes_{\F}\! \Lambda$ we have natural positive and negative $\F$-linear subspaces
\[
W_{\pm}:=\Lambda_{\pm}\cdot W^0\subset W\! \otimes_{\F}\! \Lambda,
\]
together with the associated $\F$-linear "projection" maps
\[
\id_W \otimes_{\F} P_{\pm}:W\! \otimes_{\F}\! \Lambda\to W^{\pm},
\]
which we (by abuse of notation) continue to denote by $P_{\pm}$. Suppose now that $(W,\partial)$ is a finite dimensional differential $\F$-vector space. Denoting by $\partial':=\partial \! \otimes_{\F}\! \id_{\Lambda}$ the induced differential on $W\! \otimes_{\F}\! \Lambda$ we have a free and finitely generated $\Lambda$-differential module $(W\! \otimes_{\F}\! \Lambda, \partial')$. Following Chekanov \cite{Chekanov98} we say that two $\Lambda$-linear maps 
\[
f,g:W\! \otimes_{\F}\! \Lambda \to W\! \otimes_{\F} \! \Lambda
\]
are $\lambda$-homotopic if there exists a $\Lambda$-linear map $h:W\! \otimes_{\F}\! \Lambda \to W\! \otimes_{\F} \! \Lambda$ such that 
\begin{equation}
\label{chek1}
P_-(f-g-h \partial' -\partial' h)P_+=0 
\end{equation}
as a map $W\! \otimes_{\F}\! \Lambda \to W\! \otimes_{\F} \! \Lambda$. The version of Chekanov's homotopy lemma which we need is the following. Chekanov's original formulation seems to differ slightly from the one we use here, but his proof easily carries over to our setup.
\begin{lemma}[\cite{Chekanov98}]
\label{lemchek}
Denote by $N$ a free, finitely generated $\Lambda$-module and by $(W,\partial)$ a finite dimensional differential $\F$-vector space. If there exist $\Lambda$-linear maps $\Phi:W\! \otimes_{\F}\! \Lambda \to N$ and $\Psi:N \to W\! \otimes_{\F}\! \Lambda$ such that $\Psi \Phi$ is $\lambda$-homotopic to the identity then 
\[
\dim_{\F}H(W,\partial)\leq \ran_{\Lambda}(N).
\]
\end{lemma} 


\subsubsection{An $\tilde{\omega}$-homotopy}
Viewed through (\ref{eq45}) $\Gamma$ will play the role of $A$ above. So, $\Lambda=\F[\Gamma]$ and $\lambda$ is simply given by $\tilde{\omega}:\Gamma \to \R$. We point out now that, with coefficients in a field $\F$, (\ref{eq43}) translates into an isomorphism of $\Lambda$-modules
\[
CF_{0}(F:V;\F) \otimes_{\F}\Lambda \cong CF(F:V;\F)
\]

Fix now $C,\epsilon >0$ and choose two functions $\varphi_1,\varphi_2 \in C^{\infty}(\R;[0,1])$ satisfying
\[
\varphi_1(y)= 
\left\{
\begin{array}{ll}
1, & \text{for}\ |y|<C \\
0, & \text{for}\ |y|\geq C+1
\end{array}
\right.
\quad
\&
\quad 
\varphi_2(x)= 
\left\{
\begin{array}{ll}
1, & \text{for}\ x<R \\
0, & \text{for}\ x\geq R+C
\end{array}
\right.
\]
as well as $|\varphi_2'(x)|\leq \epsilon$ for all $x\in \R$ and consider $H\in C^{\infty}_c(\tilde{M})$ defined by\footnote{Recall the choice of $\beta\in C^{\infty}_c(\R)$ made in Section \ref{secshape}.}
\[
H(x,y,p)=\left(\int_{-\infty}^x -\beta(s) ds \right)\varphi_1(y)\varphi_2(x).
\]
We then define the time-dependent and \emph{compactly supported} Hamiltonian $\tilde{H}\in C^{\infty}_c([0,1]\times \tilde{M})$ by $\tilde{H}_t(z,p)=H(\phi_F^{1-t}(z,p))$ and note that $\tilde{H}$ has Hofer norm
\[
|\! |\tilde{H}|\! |=|\! |H|\! |\leq \mathcal{S}(V)+\delta.
\]
The time-dependent Hamiltonian $G_t(z,p):=F(z,p)+\tilde{H}_t(z,p)\in C^{\infty}([0,1]\times \tilde{M})$ generates the flow $\phi_G^t=\phi_F^{t-1}\phi_H^t \phi_F^1$. Since $\phi_G^1= \phi_H^1\phi_F^1$ it is easy to see from the choices made in Section \ref{secshape} that $CF(G:V;\F)$ is a $\Lambda$-module of rank 
\[
\ran_{\Lambda}CF(G:V;\F)=\sum_{1\leq i <j\leq m}\#(L_i\cap L_j) + \sum_{1\leq i <j\leq m'}\#(L'_i\cap L'_j)
\]
provided that $C>0$ is chosen large enough and $\epsilon>0$ small enough. Therefore (\ref{eq3}) follows from Lemma \ref{lemchek} and
\begin{prop}
\label{prop1}
	There exist $\Lambda$-linear maps 
	\begin{align*}
	\Phi:&CF(F:V;\F) \to CF(G:V;\F) \\
	\Psi:& CF(G:V;\F) \to CF(F:V;\F) 
	\end{align*}
	whose composition $\Psi \Phi$ is $\tilde{\omega}$-homotopic to the identity.
\end{prop}
\begin{proof}[Proof of Proposition \ref{prop1}]
The following is basically Chekanov's proof from \cite{Chekanov98} (see also \cite{Oh97}). Chekanov constructed $\Psi$, $\Phi$ together with a suitable $\tilde{\omega}$-homotopy using Floer's continuation principle. Fix two monotone functions $\rho_{\pm}\in C^{\infty}(\R ;[0,1])$ with
\[
\rho_+(s)=\left\{
\begin{array}{ll}
0, & \text{if}\ s\leq -1 \\
1, & \text{if}\ s\geq 1
\end{array}
\right.
\quad \& \quad 
\rho_-(s)=\left\{
\begin{array}{ll}
1, & \text{if}\ s\leq -1 \\
0, & \text{if}\ s\geq 1.
\end{array}
\right.
\]
Consider also the positive and negative parts of $|\! |\tilde{H}|\! |$
\[
b_+:=\int_0^1 \max_{\tilde{M}}(\tilde{H}_t)\ dt \geq 0 \quad \& \quad b_-:=\int_0^1 \min_{\tilde{M}}(\tilde{H}_t)\ dt\leq 0,
\]
so that $|\! |\tilde{H}|\! |=b_+-b_-$ and choose $\tilde{R}>R$ such that $\supp(\tilde{H}_t)\subset \tilde{B}\times M$ for all $t\in [0,1]$, where $\tilde{B}:=[-\tilde{R},\tilde{R}]$. Now consider for $u\in C^{\infty}(Z;\tilde{M})$ the problem
\begin{equation}
\tag{$P_{\pm}$}
\label{eq13}
\left\{
\begin{array}{l}
\partial_su +\tilde{I}^{\pm}_{(s,t)}(u)(\partial_tu-X_F(u)-\rho_{\pm}(s)X_{\tilde{H}_t}(u))=0 \\
u(\R \times \{0,1\})\subset V
\end{array}
\right.
\end{equation}
where $\{\tilde{I}^{\pm}_z\}_{z\in Z}$ is a smooth $Z$-family of $\tilde{\omega}$-compatible almost complex structures satisfying 
\begin{equation}
\label{eq101}
(\phi_F^t)^*\tilde{I}^{\pm}_{(s,t)}|_{(\R^2 \backslash \tilde{B})\times M}= (\mathfrak{i}\oplus I^{\pm}_{(s,t)})|_{(\R^2 \backslash \tilde{B})\times M}
\end{equation}
for some $Z$-family of $\omega$-compatible almost complex structures $\{ I^{\pm}_{z} \}_{z\in Z}$ on $M$. We additionally require that there exists a constant $0<C<\infty$ such that
\begin{equation}
	\label{eq20}
	\tilde{I}^{\pm}_{(s,t)}=\left\{
	\begin{array}{ll}
	\tilde{J}'_t, & \text{if $\pm s<-C$ and/or $t\in \{0,1\}$} \\
	\tilde{J}^{\infty}_t, & \text{if $\pm s>C$ and/or $t\in \{0,1\}$},
	\end{array}
	\right.
\end{equation}
where $\{\tilde{J}^{\infty}_t\}_{t\in [0,1]} \in \tilde{\mathcal{J}}_F(\tilde{B})$ satisfies the condition that $\tilde{J}^{\infty}_t=\tilde{J}'_t$ for $t\in \{0,1\}$. Exactly as in Lemma \ref{lem1} one sees that $\{\tilde{I}^{\pm}_z\}_{z\in Z}$ may be chosen such that it is regular in the usual sense of Floer theory and such that 
\begin{equation}
\label{eq14}
\Delta <A(\tilde{M},V,\tilde{I}^{\pm}_z)\quad \forall \ z\in Z 
\end{equation}
We will therefore assume that this is the case from now on. Given $\widetilde{\gamma}_-\in \Crit(\mathcal{A}_{F:V})$ and $\widetilde{\gamma}_+\in \Crit(\mathcal{A}_{G:V})$ we define the length
\[
l_+(\widetilde{\gamma}_-,\widetilde{\gamma}_+):=\mathcal{A}_{F:V}(\widetilde{\gamma}_-)-\mathcal{A}_{G:V}(\widetilde{\gamma}_+)\in \R
\]
and denote by $\mathcal{M}^+(\widetilde{\gamma}_-,\widetilde{\gamma}_+)$ the space of finite energy solutions $u$ to ($P_+$) enjoying the asymptotic conditions $\lim_{s\to -\infty}u_s=\widetilde{\gamma}_-$ and $\lim_{s\to \infty}u_s=\widetilde{\gamma}_+$ in $\Omega_V$. Since $\tilde{I}^+$ is regular $\mathcal{M}^+(\widetilde{\gamma}_-,\widetilde{\gamma}_+)$ is a smooth manifold. Given $u\in \mathcal{M}^+(\widetilde{\gamma}_-,\widetilde{\gamma}_+)$ one integrates by parts to see that 
\begin{equation}
\label{eq15}
E_{\tilde{I}^+}(u)= l_+(\widetilde{\gamma}_-,\widetilde{\gamma}_+)+\int_{-\infty}^{\infty}\int_0^1 \dot{\rho}_+(s)\tilde{H}_t(u)\ dtds.
\end{equation}
If $l_+(\widetilde{\gamma}_-,\widetilde{\gamma}_+)\leq \delta -b_-$ it follows from this and monotonicity of $\rho_+$ that 
\begin{align*}
E_{\tilde{I}^+}(u)\leq \delta -b_- +\int_{0}^{1}\max_{\tilde{M}}(\tilde{H}_t) \ dt= |\! | \tilde{H}|\! | +\delta \leq \mathcal{S}(V)+2\delta = \Delta, 
\end{align*}
for every $u\in \mathcal{M}^+(\widetilde{\gamma}_-,\widetilde{\gamma}_+)$. In particular, in this case, (\ref{eq14}) implies that no bubbleing occurs along $\mathcal{M}^+(\widetilde{\gamma}_-,\widetilde{\gamma}_+)$. If in addition $\dim \mathcal{M}^+(\widetilde{\gamma}_-,\widetilde{\gamma}_+)=0$ it follows from regularity of $\tilde{I}^+$ that $\mathcal{M}^+(\widetilde{\gamma}_-,\widetilde{\gamma}_+)$ is compact. Hence, we can define $\Phi:CF(F:V;\F)\to CF(G:V;\F)$ as the unique $\Lambda$-linear operator whose $(\widetilde{\gamma}_-,\widetilde{\gamma}_+)$'th matrix element equals 0 if $\dim \mathcal{M}^+(\widetilde{\gamma}_-,\widetilde{\gamma}_+)\neq 0$ or $l_+(\widetilde{\gamma}_-,\widetilde{\gamma}_+)> \delta -b_-$ and otherwise equals
\begin{equation}
\label{eq17}
\sum_{u\in \mathcal{M}^+(\widetilde{\gamma}_-,\widetilde{\gamma}_+)}C(u)\otimes_{\Z}\id_{\F}:C(\widetilde{\gamma}_-)\otimes_{\Z}\F \to C(\widetilde{\gamma}_+)\otimes_{\Z}\F
\end{equation}
as defined in \cite[Section 3.8.1]{Zapolsky15}. Similarly, given $\widetilde{\gamma}_- \in \Crit(\mathcal{A}_{G:V})$ and $\widetilde{\gamma}_+ \in \Crit(\mathcal{A}_{F:V})$ we consider the quantity 
\[
l_-(\widetilde{\gamma}_-,\widetilde{\gamma}_+):=\mathcal{A}_{G:V}(\widetilde{\gamma}_-)-\mathcal{A}_{F:V}(\widetilde{\gamma}_+),
\]
and denote by $\mathcal{M}^-(\widetilde{\gamma}_-,\widetilde{\gamma}_+)$ the space of finite energy solutions $u$ to ($P_-$) satisfying $\lim_{s\to -\infty}u_s=\widetilde{\gamma}_-$ and $\lim_{s\to \infty}u_s=\widetilde{\gamma}_+$ in $\Omega_V$. Again $\mathcal{M}^-(\widetilde{\gamma}_-,\widetilde{\gamma}_+)$ is a smooth manifold and integration by parts yields  
\begin{equation}
\label{eq16}
E_{\tilde{I}^-}(u)=l_-(\widetilde{\gamma}_-,\widetilde{\gamma}_+)+\int_{-\infty}^{\infty}\int_0^1 \dot{\rho}_-(s)\tilde{H}_t(u)\ dtds.
\end{equation}
for every $u\in \mathcal{M}^-(\widetilde{\gamma}_-,\widetilde{\gamma}_+)$. In particular if $l_-(\widetilde{\gamma}_-,\widetilde{\gamma}_+)\leq \delta +b_+$ then 
\begin{align*}
E_{\tilde{I}^{-}}(u)\leq \delta+b_+  - \int_{0}^{1}\min_{\tilde{M}}(\tilde{H}_t) \ dt\leq \Delta,
\end{align*}
for every $u\in \mathcal{M}^-(\widetilde{\gamma}_-,\widetilde{\gamma}_+)$, so no bubbleing occurs along $\mathcal{M}^-(\widetilde{\gamma}_-,\widetilde{\gamma}_+)$. We can therefore define $\Psi :CF(G:V;\F)\to CF(F:V;\F)$ as the unique $\Lambda$-linear map whose $(\widetilde{\gamma}_-,\widetilde{\gamma}_+)$'th matrix element equals 0 if $\dim \mathcal{M}^-(\widetilde{\gamma}_-,\widetilde{\gamma}_+) \neq 0$ or $l_-(\widetilde{\gamma}_-,\widetilde{\gamma}_+)> \delta +b_+$ and otherwise equals
\[
\sum_{u\in \mathcal{M}^-(\widetilde{\gamma}_-,\widetilde{\gamma}_+)} C(u)\otimes_{\Z}\id_{\F}: C(\widetilde{\gamma}_-)\otimes_{\Z}\F \to C(\widetilde{\gamma}_+)\otimes_{\Z}\F
\]

The aim now is to construct an $\tilde{\omega}$-homotopy from $\Psi \Phi$ to the identity. I.e. we need to construct a $\Lambda$-linear map $h:CF(F:V;\F)\to CF(F:V;\F )$ such that 
\begin{equation}
\label{eq19}
P_-(\id - \Psi \Phi -h\partial -\partial h)P_+=0
\end{equation}
In order to construct $h$ we choose a function $\rho \in C^{\infty}([0,\infty)\times \R; [0,1])$, written $(\tau,s)\mapsto \rho_{\tau}(s)$, such that for every $\tau \in [0,\infty)$ the function $\rho_{\tau}\in C^{\infty}_c(\R ;[0,1])$ satisfies the monotonicity condition  
\begin{equation}
\label{eq23}
\frac{d\rho_{\tau}}{d s}(s)\left\{
\begin{array}{ll}
\geq 0, & \text{if}\ s\leq 0  \\
\leq 0, & \text{if}\ s\geq 0 
\end{array}
\right.
\end{equation}
Moreover, we require the condition that $[0,\infty)\ni \tau \mapsto \rho_{\tau}(0)$ is a monotone function onto $[0,1]$ as well as the condition that, for $\tau \geq 2$, we have 
\[
\rho_{\tau}(s)=\left\{
\begin{array}{ll}
\rho_+(s+\tau), & \text{if} \ s\leq 0 \\
\rho_-(s-\tau), & \text{if} \ s\geq 0.
\end{array}
\right.
\]
Consider for $\tau \in [0,\infty)$ the problem
\begin{equation}
\label{eq21}
\left\{
\begin{array}{l}
\partial_su +\tilde{I}^{\tau}_{(s,t)}(u)(\partial_tu-X_F(u)-\rho_{\tau}(s)X_{\tilde{H}_t}(u))=0 \\
u(\R \times \{0,1\})\subset V,
\end{array}
\right.
\end{equation}
where $\{\tilde{I}^{\tau}_{z}\}_{(\tau,z)\in [0,\infty) \times Z}$ is a family of $\tilde{\omega}$-compatible almost complex structures satisfying (\ref{eq101}) (with $\pm$ replaced by $\tau$) for a family of $\omega$-compatible almost complex structures $\{I^{\tau}_{z}\}_{(\tau,z)\in [0,\infty) \times Z}$ on $M$. We require that  
\[
\tilde{I}^{\tau}_{(s,t)}=\tilde{J}'_{t} \quad \text{if $\tau =0$ and/or $t\in \{0,1\}$}.
\]
We also require the existence of a constant $0<C<\infty$ such that for all $\tau \geq C$
\[
\tilde{I}^{\tau}_{(s,t)}=\left\{
\begin{array}{ll}
\tilde{I}^+_{(s+\tau ,t)}, & \text{if} \ s\leq 0 \\
\tilde{I}^-_{(s-\tau ,t)}, & \text{if} \ s\geq 0,
\end{array}
\right.
\]
and for all $\tau \leq C$
\[
\tilde{I}^{\tau}_{(s,t)}=\tilde{J}'_t \quad \text{if $s$ is sufficiently large.}
\]
As above $\{ \tilde{I}^{\tau}_{z} \}_{(\tau,z)}$ can be chosen to be regular and satisfy
\begin{equation}
\label{eq26}
\Delta < A(\tilde{M},V,\tilde{I}^{\tau}_{z}) \quad \forall \ (\tau,z) \in [0,\infty)\times Z.
\end{equation}
Given $\tau \in [0,\infty)$ and $\widetilde{\gamma}_-,\widetilde{\gamma}_+ \in \Crit(\mathcal{A}_{F:V})$ we denote by $\mathcal{M}^{\rho_{\tau}}(\widetilde{\gamma}_-,\widetilde{\gamma}_+)$ the space of all finite energy solutions $u\in C^{\infty}(Z,\tilde{M})$ to (\ref{eq21}) satisfying $\lim_{\tau \to -\infty}u_s=\widetilde{\gamma}_-$ and $\lim_{\tau \to \infty}u_s=\widetilde{\gamma}_+$ in $\Omega_V$ and we define
\[
\mathcal{M}^{\rho}(\widetilde{\gamma}_-,\widetilde{\gamma}_+):=\{ (\tau,u)\ :\ \tau \in [0,\infty), \ u\in \mathcal{M}^{\rho_{\tau}}(\widetilde{\gamma}_-,\widetilde{\gamma}_+) \}.
\]
Since $\{ \tilde{I}^{\tau}_{(s,t)} \}_{(\tau,s,t)}$ is regular, $\mathcal{M}^{\rho}(\widetilde{\gamma}_-,\widetilde{\gamma}_+)$ is a smooth manifold for every $\widetilde{\gamma}_-,\widetilde{\gamma}_+\in \Crit(\mathcal{A}_{F:V})$ and integration by parts yields 
\begin{equation}
\label{eq22}
E_{\tilde{I}^{\tau}}(u)=l(\widetilde{\gamma}_-,\widetilde{\gamma}_+)+\int_{-\infty}^{\infty}\int_{0}^{1}\frac{d\rho_{\tau}}{ds}(s)\tilde{H}_t(u)\ dtds \quad \forall \ u\in \mathcal{M}^{\rho_{\tau}}(\widetilde{\gamma}_-,\widetilde{\gamma}_+).
\end{equation}
In particular we see that if $l(\widetilde{\gamma}_-,\widetilde{\gamma}_+)\leq \delta$ then 
\begin{align*}
	E_{\tilde{I}^{\tau}}(u)&\leq \delta + \int_{-\infty}^{0}\int_{0}^{1}\frac{d\rho_{\tau}}{ds}(s)\tilde{H}_t(u)\ dtds + \int_{0}^{\infty}\int_{0}^{1}\frac{d\rho_{\tau}}{ds}(s)\tilde{H}_t(u)\ dtds \\
	&\leq \delta + \rho_{\tau}(0)b_+ -\rho_{\tau}(0)b_- =\delta +\rho_{\tau}(0)|\! |\tilde{H}|\! | \leq \Delta \quad \forall \ u\in \mathcal{M}^{\rho_{\tau}}(\widehat{\gamma}_-,\widehat{\gamma}_+)
\end{align*}
so, in this case, (\ref{eq26}) implies that no bubbleing occurs along $\mathcal{M}^{\rho}(\widetilde{\gamma}_-,\widetilde{\gamma}_+ )$. If in addition $\dim \mathcal{M}^{\rho}(\widetilde{\gamma}_-,\widetilde{\gamma}_+ )=0$ then $\#\mathcal{M}^{\rho}(\widetilde{\gamma}_-,\widetilde{\gamma}_+ )<\infty$, so we can define $h:CF(F:V)\to CF(F:V)$ as the unique $\Lambda$-linear map whose $(\widetilde{\gamma}_-,\widetilde{\gamma}_+)$'th matrix element equals 0 if $\dim \mathcal{M}^{\rho}(\widetilde{\gamma}_-,\widetilde{\gamma}_+ )\neq 0$ or $l(\widehat{\gamma}_-,\widehat{\gamma}_+)> \delta$ and otherwise equals
\[
\sum_{(\tau, u)\in \mathcal{M}^{\rho}(\widetilde{\gamma}_-,\widetilde{\gamma}_+)} C(u)\otimes_{\Z} \id_{\F}:C(\widetilde{\gamma}_-)\otimes_{\Z}\F\to C(\widetilde{\gamma}_+)\otimes_{\Z}\F.
\]
To finish the proof we need to check that (\ref{eq19}) is satisfied. To do that we fix $\widetilde{\gamma}_-=(q_-,\widehat{q}_-), \widetilde{\gamma}_+=(q_+,\widehat{q}_+) \in \Crit(\mathcal{A}_{F:V})$ with $\tilde{\omega}(\widehat{q}_-)\geq 0$ and $\tilde{\omega}(\widehat{q}_+)\leq 0$ and we need to check that the $(\widetilde{\gamma}_-,\widetilde{\gamma}_+)$'th matrix element of the operator 
\begin{equation}
\label{eq46}
\id - \Psi \Phi - h\partial -\partial h
\end{equation}
equals 0. This is clearly the case if the Conley-Zehnder indices of $\widetilde{\gamma}_-$ and $\widetilde{\gamma}_-$ differ, so we only consider the case when these coincide, in which case $\dim \mathcal{M}^{\rho}(\widetilde{\gamma}_-,\widetilde{\gamma}_+)=1$. Since
\[
l(\widetilde{\gamma}_-,\widetilde{\gamma}_+)=f(q_-)-\tilde{\omega}(\widehat{q}_-)-f(q_+)+\tilde{\omega}(\widehat{q}_+)\leq f(q_-)-f(q_+) \leq \delta, 
\]
no bubbling occurs along $\mathcal{M}^{\rho}(\widetilde{\gamma}_-,\widetilde{\gamma}_+)$, so it is compact up to Floer breaking. By the usual gluing argument every configuration counted in the $(\widetilde{\gamma}_-,\widetilde{\gamma}_+)$'th matrix element of (\ref{eq46}) occurs as a boundary point of the compactification of $\mathcal{M}^{\rho}(\widetilde{\gamma}_-,\widetilde{\gamma}_+)$ and it therefore follows as in \cite[Section 3.8.2]{Zapolsky15} that the $(\widetilde{\gamma}_-,\widetilde{\gamma}_+)$'th matrix element of (\ref{eq46}) equals 0 if only we argue that every boundary point of the compactification of $\mathcal{M}^{\rho}(\widetilde{\gamma}_-,\widetilde{\gamma}_+)$ occurs in (\ref{eq46}). This follows from the estimates on the last page of \cite{Chekanov98}.  
\end{proof}

\bibliographystyle{plain}
\bibliography{BIB}
\end{document}